\documentclass{amsart}

\usepackage{graphicx,afterpage,float,lscape, caption}

\usepackage{amsmath}
\usepackage{amsfonts}
\usepackage{amssymb}
\usepackage{amsthm}

\usepackage{algorithm}
\usepackage{mathrsfs}
\usepackage{wrapfig}
\usepackage{tikz}
\usetikzlibrary{matrix,arrows,decorations.pathmorphing}
\usepackage{url}

\newcommand{\bb}[1]{\mathbb{#1}}
\newcommand{\Z}[1]{\bb{Z}_{#1}}
\newcommand{\F}[1]{\bb{F}_{#1}}

\newcommand{\<}[1]{\langle #1 \rangle}

\newcommand{\ceiling}[1]{\lceil #1 \rceil]}
\newcommand{\FBar}[1]{\overline{\mathbb{F}}_{#1}}

\newcommand{\fail}{\texttt{FAIL}}

\DeclareMathOperator{\ppd}{ppd}
\DeclareMathOperator{\GL}{GL}

\DeclareMathOperator{\SL}{SL}

\DeclareMathOperator{\Sp}{Sp}

\DeclareMathOperator{\SU}{SU}

\DeclareMathOperator{\SO}{SO}

\DeclareMathOperator{\App}{App}

\DeclareMathOperator{\quo}{quo}

\DeclareMathOperator{\res}{res}

\DeclareMathOperator{\M}{M}

\DeclareMathOperator{\Class}{Class}
\newtheorem{theorem}{Theorem}[section]
\newtheorem{corollary}[theorem]{Corollary}
\newtheorem{conjecture}[theorem]{Conjecture}

\newtheorem{lemma}[theorem]{Lemma}
\newtheorem{proposition}[theorem]{Proposition}

\theoremstyle{definition}
\newtheorem{definition}[theorem]{Definition}

\newtheorem{remark}[theorem]{Remark}

\title{A Las Vegas Rewriting Algorithm for the Symmetric Square Representation of Classical Groups}

\author[B.~Corr]{Brian P. Corr}
\address{Brian P. Corr, \newline 
Centre for Mathematics of Symmetry and Computation,\newline
School of Mathematics and Statistics,\newline
The University of Western Australia,
 Crawley, WA 6009, Australia}
\curraddr{Departamento de Matem\'{a}tica, \newline
Instituto de Ci\^{e}ncias Exatas,\newline 
Universidade Federal de Minas Gerais, \newline
Av. Ant\^{o}nio Carlos, 6627, 31270-901, \newline
Belo Horizonte, MG, Brazil}
\email{brian.p.corr@gmail.com}

\thanks{This research forms part of the ARC Discovery Project DP110101153. The author was supported by an Australian Postgraduate Award, a UWA Top-Up Scholarship, an Australian Mathematical Society Lift-Off Fellowship and by cNPQ and CAPES. The author wishes to thank Cheryl Praeger and \`Akos Seress for their support and input.}

\dedicatory{Dedicated to the memory of \`Akos Seress, who provided a great deal of input into this work.}
\begin{document}
\begin{abstract}
In constructive recognition of a representation of a Classical group $G$, much attention has been paid to the natural representation as well as to generic (Black Box) algorithms that treat all representations uniformly. There are theoretical and practical improvements to be made by giving special treatment to certain non-natural representations that arise frequently. In this paper we present and analyse a Las Vegas algorithm for rewriting the Symmetric Square representation.
\end{abstract}
\maketitle
\section{Introduction}
A major goal of Computational Group Theory is the solving of the \emph{constructive recognition problem}, which asks for fast (that is, polynomial time where possible) algorithms for the following tasks:
\begin{enumerate}
\item Given a group $G < H$ input into a computer in some arbitrary way, determine the isomorphism type of $G$ (\emph{nonconstructive recognition}); and
\item Produce an isomorphism from $G$ into some `standard copy' of this type of group, and provide a scheme for, given $g$ in the `ambient' group $H$, deciding if $g\in G$, and if so, rewriting $g$ in this `standard form' (\emph{constructive recognition}).
\end{enumerate} 
The most common ways of inputting a group into a computer are as a set of generators and relations, or as a set of generating permutations or matrices: much effort has been spent in dealing with each of these representations separately, as well as in dealing with \emph{Black Box Groups}, a theoretical setting in which no structural information about the way in which the group is represented is assumed.\\\\
Black Box algorithms provide complete generality, and hence apply in all settings: in particular, if the representation of $G$ in the computer does not offer much information, then a Black Box algorithm will approach `maximal effectiveness'. On the other hand, particularly natural representations of a group (for example, the representation of the Symmetric Group $S_n$ as permutations of $n$ points, or the natural representation of the General Linear Group) can be dealt with much more quickly and effectively using methods specific to the representation.\\\\
The `Composition Tree' framework \cite{leedham2001computational,neunhoffer2006data} provides an elegant method for dealing with arbitrary matrix groups: using various methods, beginning with the \texttt{MEAT-AXE} procedure of Holt \& Rees \cite{holt1996testing}, the input group $G$ is searched for normal subgroups $N$, and a structure is set up so that $N$ and $G/N$ may be dealt with separately. This process, applied recursively, yields a binary rooted tree that gives the procedure its name.\\\\
As with many group-theoretic frameworks relying on normal subgroups, the process terminates when $G$ is almost simple (at the leaves of the tree). Each almost simple group presents its own unique challenges, and each family of almost simple groups is dealt with separately. In the matrix group setting, the Classical groups have received a great deal of attention, beginning with the Neumann-Praeger nonconstructive $\SL$-recognition algorithm \cite{neumann1992recognition}: the problem has essentially been solved in the Black Box cases (which make no attempt to exploit the geometry of the situation) and in the natural representation (where the geometry is most rich): see \cite{obrien2006towards} for a survey. Attention is now paid to the remaining representations for which there is still meaningful geometric information to use.\\\\
In this paper we provide an updated and corrected version of the Magaard-O'Brien-Seress algorithm for constructively recognising the Special Linear Group in its action on the Symmetric Square module, and apply similar methods to constructively recognise all Classical groups (Unitary, Symplectic and Orthogonal) in their actions on the unique irreducible $FG$-module of dimension $n$, where $\binom{d+1}{2} -2\leq n\leq \binom{d+1}{2}$ (in practice, this procedure will work perfectly well when the module is, in fact, the Symmetric Square, though in some cases the Symmetric Square is reducible). We wish to acknowledge and thank Cheryl Praeger and \`Akos Seress for their support, expertise and advice during the preparation of this paper.

\begin{theorem}\label{thm: main}
Let $X \subseteq \GL(n,q)$ be a set of matrices generating a classical group $G= \Class(d,q)$, such that the module $W$ defined by the action of $H = \<{X}$ is an irreducible section of the Symmetric Square module $S^2(V)$ of codimension at most 2. Let $d'$ be as in Table \ref{SpecialTable}, and suppose that $G \not\in \{\Sp(d,3), \SO^{\epsilon}(d,3)\}$. Then assuming that Conjecture \ref{Division By Zero Probability Symmetric Square Case} holds in the Symmetric Square case, and excluding some small values of $d$ (see Table \ref{MinDTable}), there exists a Las Vegas algorithm which, with probability at least $1-\epsilon$, sets up a data structure for rewriting $H$ as a projective representation in its natural dimension, with complexity 
\[
O
\left(
\xi_H  d^2 \log^2 q \log \epsilon^{-1}
+
\rho_{q} \left( d^9 \log d\log\log d \log q  +  d^8 \log d\log \log d \log^3 q  \log \epsilon^{-1}  \right)
\right),
\]
where $\xi_H$ is the cost of choosing a random element of $H$, and $\rho_q$ is the cost of a field operation in $\F{q}$. Once the initialisation procedure is complete, there is a Las Vegas algorithm for rewriting a group element (that is, returning a $d\times d$ matrix) with complexity 
\[
O
\left(
(\xi_{H} + \rho_{q} d^8 \log d \log\log d \log q ) )\log \epsilon^{-1} 
\right).
\]
\end{theorem}

\begin{table}
\begin{center}
\begin{tabular}{|c|c|c|c|}
\hline
\hline
$G$    &  Minimum $d$ &   Conditions  \\
\hline
\hline
$\SL(d,q)$  &  			3		&   -- \\
\hline
$\SU(d,q)$  &  			3		&    $d$ odd \\
            &  			4		&  	  $d$ even\\
\hline 
$\Sp(d,q)$  & 		 	6		&  	  $d$ even \\
\hline 
$\SO^{-}(d,q)$ & 		6 		&  	  $d$ even \\
\hline 
$\SO^{\circ}(d,q)$  &  	7		&  	  $d$ odd \\
\hline 
$\SO^{+}(d,q)$  &  		8		&  	  $d$ even \\
\hline
\hline
\end{tabular}
\caption{Minimum values of $d$ for \texttt{Initialise}}\label{MinDTable}
\end{center}
\end{table}
We prove Theorem \ref{thm: main} over the course of the paper, by describing explicitly the steps of the algorithm. This is a very specialised algorithm which provides a major improvement over the runtimes of the existing best algorithms (although the existing algorithms remain extremely useful, for they apply in many more cases than this one).
\section{Modules and Representations}\label{sec: Modules and Reps}
In this section we introduce some notation, in particular the Symmetric Square module and its irreducible constituents (note that in many cases, the Symmetric Square is itself irreducible). Let $V= V(d,q)$ be a vector space over a field $F=\F{q}$ of order $q$. Then $V$ is called an $FG$-module if the group $G$ acts on $V$ in a way compatible with the vector space structure of $V$: that is, if $(v+w)^g = v^g + w^g$ and $(av)^g = av^g$ for all $g\in G, v,w\in V, a\in F$. An $FG$-submodule is a subspace of $V$ left invariant by the action of $G$: an \emph{irreducible $FG$-module} is a module with no proper nontrival submodules.
When a group $G$ acts on several $FG$-modules, we use a subscript where necessary to distinguish the actions (for example, $g_V$ denotes the action of $g\in G$ on an $FG$-module $V$).
For two $FG$-modules $V,W$ with bases $\{v_1, \ldots, v_{d_1}\}, \{w_1, \ldots , w_{d_2}\}$, the \emph{tensor product} $V\otimes W$ is the $FG$-module with basis $\{ v_i \otimes v_j \mid 1\leq i\leq d_1, 1\leq j\leq d_2\}$: an element $g\in G$ acts in the diagonal way $(v\otimes w)^g = v^g \otimes w^g$ (extending by linearity, this gives an $FG$-module structure).\\\\
Consider an extension $K := \F{q^{d'}}$ of $F := \F{q}$, and fix a basis $\{v_1, \ldots v_d\}$ of $V$. Then viewing $K$ as an $F$-vector space (and, in turn, an $FG$-module with $G$ acting trivially), the tensor product $V\otimes K$ is isomorphic to the $K$-vector space with basis $\{v_1 , \ldots , v_d\}$, with the same $G$-action. We denote this module by $V_K$, and observe that all properties of $g\in G$ carry over in this action: in particular, the characteristic polynomial does not change, although its irreducible factors do, since the notion of irreducibility of a polynomial depends upon the field. By considering the action of $g$ on $V_K$, we may access a richer eigenstructure.
\subsection{The Symmetric Square $S^2(V)$}\label{sec:Symmetric Square}
The Symmetric Square module $S^2(V)$ is an irreducible constituent of the tensor square $V\otimes V$: let $G\in \GL(V)$, and let $\{v_i \mid 1 \leq i \leq d \}$ be a basis for $V$. The \emph{Symmetric Square Module} $S^2(V)$ is the $FG$-submodule of $V\otimes V$ generated by
\[
\{ v_i \otimes v_i \mid 1 \leq i \leq d \}  \cup \{ v_i \otimes v_j + v_j \otimes v_i \mid 1 \leq i  < j \leq d \}.
\]
Since $(-v)\otimes (-w) = v\otimes w$, the element $-1 \in G$ acts trivially on $V\otimes V$ (and hence on $S^2(V)$. In fact, the set $\{\pm 1\}$ is precisely the kernel of this action. For this reason, our rewriting algorithm can only return the action on $V$ modulo this kernel.
Our primary method for the rewriting algorithm is the analysis of the eigenvalues and eigenspaces of a group element. Since the eigenstructure of a group element depends on its action on a vector space, an element $g$ will usually have different eigenstructures, depending on whether we consider $g_V, g_{V\otimes V}, g_{S^2(V)}$ or an action on another module, and also depending on the underlying field (note that when the field changes, the characteristic polynomial will not change: however, its roots might!).
We now present several relationships between the eigenstructures of an element $g\in G$ in its action on different modules. Let $F=\F{q}$, let $K = \F{q^{d'}}$ for some integer $d'$, let $V=F^d$, and suppose that $g\in \GL(d,q)$ has $d$ (not necessarily distinct) eigenvalues $\{\lambda_{i} \mid 1\leq i\leq d\}$ in its action on $V_K$, and there exists a basis $\{v_1, \ldots,v_d\}$ for $V_K$ of $g$-eigenvectors (so that for each $i$, $v_ig=\lambda_i v_i$). Then the eigenvalues of $g$ in its action on $(V\otimes_F V)_K$ are
\[
\{\lambda_i\lambda_j\mid 1\leq i \leq j\leq d\},
\]
and for each $i,j$, both $v_{i}\otimes v_j$ and $v_j\otimes v_i$ are $(\lambda_{i}\lambda_j)$-eigenvectors in $(V\otimes_F V)_K$. Moreover, these are the only eigenvalues of $g$ in $(V\otimes V)_K$. 
\begin{lemma}\label{EigenstructureLemma}
Let $g\in \GL(V)$, and suppose that $g$ has eigenvalues $\lambda_1, \ldots, \lambda_d$ in its action on $V$, and that $\{v_1, \ldots, v_d\}$ is a basis for $V$ such that for all $i$, $v_i$ is a $\lambda_i$-eigenvector for $g_{V}$. Let $v_{ij} = v_{i}\otimes v_j + v_j\otimes v_i$ when $i\neq j$, and $v_{ii}=v_i\otimes v_i$. Then 
\[
\{v_{ij} \mid 1\leq i \leq j\leq d\}
\]
is a basis for $S^2(V)$, such that for every $i,j$, $v_{ij}$ is a $(\lambda_{i}\lambda_j)$-eigenvector for $g$ in its action on $S^2(V)$.
\end{lemma}
\begin{proof}
By the comments above, both $v_i\otimes v_j, v_j\otimes v_i$ are $(\lambda_i\lambda_j)$-eigenvectors of $g_{V\otimes V}$, and so any linear combination of the two is itself a $(\lambda_i\lambda_j)$-eigenvector. That this set forms a basis is clear by comparing dimensions.
\end{proof}
We often consider the matrix of a given $g\in G$ in its action on multiple bases. For this reason we introduce the following notation: if $g\in G$ and $\mathscr{B}$ is an ordered basis for $V$, then $g_{\mathscr{B}}$ denotes the matrix of $g$ with respect to $\mathscr{B}$; if $b_i,b_j\in \mathscr{B}$ then and $g_{b_1b_2}$ denotes the coefficient of $b_2$ in the expansion of ${b_1}^g$ (in the case that our basis is indexed in the usual way, this is the $(i,j)$-entry of $g_{\mathscr{B}}$). \\\\
Let $V$ be an $FG$-module, and let $\mathscr{B}:= \{b_i \mid 1\leq i\leq d\}, \mathscr{B}'= \{b_i' \mid 1\leq i\leq d\}$ be bases for $V$, such that for every $i$, there exists $c_i\in F^* := F\setminus \{0\}$ such that $b_i' = c_ib_i$. Then for every $g\in G$ and for all $i,j$, we have $g_{b_i'b_j'} = \frac{c_i}{c_j} g_{b_i b_j}$. If $\varphi:V\to W$ is an isomorphism of $FG$-modules and $\mathscr{V}$ is a basis for $V$, then $\mathscr{V}^{\varphi}$ is a basis for $W$, and for all $v,w\in \mathscr{V}, g\in G$, we have $g_{v^{\varphi}w^{\varphi}} = g_{vw}$. \\\\
If instead $W$ is a $G$-invariant subspace of $V$, then the quotient space $V/W := \{ v+ W \mid v\in V\}$ is an $FG$-module of dimension $\dim V - \dim W$, with the action of $g\in G$ defined by $(v+W)^g = v^g + W$. Moreover, suppose that $\quo_W: V \to V/W$ is the natural quotient map $v\mapsto v+W$, and  $e,f \in V$ are basis vectors such that $\<{e,f} \cap W = \{0\}$. Then for every $g\in G$, we have that $g_{e,f} = g_{\quo(e),\quo(f)}$. Finally suppose that $e,f \in\mathscr{V}$ are basis vectors such that $e,f \in W$. Then for every $g\in G$, we have that $g_{ef} = (g_W)_{ef}$.
%
%
%
%
Just as we may seek normal subgroups of groups by defining homomorphisms and inspecting their kernels, we will construct submodules of $FG$-modules by considering the nullspaces of certain maps: if $T$ is a $G$-invariant linear form on an $FG$-module $W$, then the kernel of $T$ is an $FG$-submodule of $W$.
\begin{lemma}\label{MissingPieces}
Let $G$ be a group, $W$ be an $FG$-module, and let $T$ be a $G$-invariant linear form on $W$ (that is, a linear map $W\to F$), and let $g\in G$. If $g_W$ has an eigenvalue $\lambda\neq 1$ in $F$, then the $\lambda$-eigenspace of $g$ is contained in the kernel of $T$.
\end{lemma}
Given the matrix of a group element $g\in G$ in its action on a module $V$ with respect to a fixed basis (as is always the case when dealing with a computer representation of a group), we may easily construct corresponding matrices for the actions of $g$ on $V\otimes V$ and $S^2(V)$:
\begin{lemma}\label{Basic Equations Lemma}
Let $G\leq \GL(V)$, and let $\mathscr{V}:=\{v_i\mid 1\leq i \leq d\}$ be a basis for $V$. Let $g\in G$, and write $g_{ij} = g_{v_iv_j}$. Define $v_{ij}$ as in Lemma \ref{EigenstructureLemma}. Then
\begin{enumerate}
\item $g_{v_{i}\otimes v_j, v_k\otimes v_{\ell} } = g_{ik}g_{j\ell}$, for any $i,j,k,\ell \in [1\ldots d]$; and
\item for $1\leq i \leq j \leq d, 1\leq k\leq \ell\leq d$, we have 
\[
g_{v_{ij},v_{k\ell}} = g_{ik}g_{j\ell} + (1-\delta_{k\ell})g_{i\ell}g_{jk}.
\]
\end{enumerate}
\end{lemma}
\begin{proof}
By definition we have 
\begin{align*}
(v_i\otimes v_j)^g 
= v_i^g \otimes v_j^g 
&= (\sum_{k=1}^d g_{ik} v_k)\otimes(\sum_{\ell=1}^{d} g_{j\ell} v_\ell) \\
&= \sum_{k=1}^d\sum_{\ell=1}^d (g_{ik}g_{j\ell} )v_k\otimes v_{\ell} ,
\end{align*}
and (i) follows. For (ii), observe that, for $i \neq j$, we have
\begin{align*}
(v_i\otimes v_j + v_j\otimes v_i)^g 
&= v_i^g\otimes v_j^g + v_j^g\otimes v_i^g \\
&= \sum_{k=1}^d\sum_{\ell=1}^d (g_{ik}g_{j\ell} + g_{jk}g_{i\ell})v_k\otimes v_{\ell} 
\end{align*}
Since switching $k,\ell$ does not change the value of $g_{ik}g_{j\ell} + g_{jk}g_{i\ell}$, we have
\[
(v_i\otimes v_j + v_j\otimes v_i)^g =  \sum_{k=1}^d\sum_{\ell \geq k}(g_{ik}g_{j\ell} + g_{jk}g_{i\ell})(v_k\otimes v_{\ell} + v_{\ell}\otimes v_k).
\]
The proof when $i = j$ follows by an identical argument.
\end{proof}
\section{Special Elements and their Eigenstructure}
\subsection{Singer Cycles (Motivation)}
In \cite{magaard2008recognition}, Magaard, O'Brien \& Seress exploit the eigenstructure of Singer Cycles in $G=\SL(d,q)$ in their action on small degree $\F{q}G$-modules to produce their algorithm for rewriting. In other Classical groups, such elements cannot always be found. A \emph{Special Element} has many of the same properties: in essence we define a Special Element as a `good enough analogue' to the elements exploited in \cite{magaard2008recognition}. Special elements act irreducibly on a subspace of $V$ of large dimension, and have large order (in both of these respects, the meaning of `large' is dependent on our needs). \\\\
Let $q$ be a prime power, and $d$ a positive integer. Then a prime $r$ is called a {\it primitive prime divisor ($ppd$)} of $q^d-1$ if $r \mid (q^d-1)$; and for $1\leq e < d$, we have $r\nmid (q^e-1)$. A \emph{Singer Cycle} in $\GL(d,q)$ is an element of order $q^d-1$: we identify such elements with primitive elements of the extension field $\F{q^d}$. For $1\leq d' \leq d$, An element $s\in \GL(d,q)$ is called a $\ppd(d,q;d')$-element if $o(s)$ is divisible by a primitive prime divisor of $q^{d'}-1$.\\\\
If $s$ is a $\ppd(d,q;d')$-element, then the characteristic polynomial $c_s(t)$ of $s$ has an irreducible divisor $f$ of degree $d'$, and acts irreducibly on a unique $d'$-dimensional subspace $V_f$ of $V$ (the $f$-primary component of $V$ \cite{hartleyhawkes}): in the case of Singer Cycles, $c_s(t)$ is irreducible and $V_f=V$. Some subgroups of $\GL(d,q)$ have no Singer Cycles, and so we must settle for $d'$ as large as possible (in the worst case, $d'=d-2$).\\\\
The fact that $c_s(t)$ has a high degree irreducible divisor may seem, at first, bad news for any attempt to exploit the eigenstructure of $s$ -- after all, eigenvalues will arise when the divisors of $c_s(t)$ have \emph{smallest} degree, not largest. However, an irreducible divisor of $c_s(t)$ of degree $d'$ gives rise to $d'$ distinct eigenvalues in the action of $s$ on $V_K = V\otimes K$, where $K=\F{q^{d'}}$. Moreover, these distinct eigenvalues (and, consequently, their eigenvectors) form an orbit of the action of the Frobenius map $\sigma:x\mapsto x^q$ of the extension $K/F$.
\begin{lemma}\label{EigenspacesOfIrreducibleElements}
If $s\in\GL(d,q)$ acts irreducibly on $V$, then the eigenvalues of $s$ on $V_K$ are 
\[
\{\ell_{i}=\lambda^{q^{i-1}} \mid 1\leq i\leq d\}
\]
for some $\lambda\in K$ with $o(\lambda)=o(s)$, and there exists a basis $\mathscr{E}(s,V):=\{e_i\mid 1\leq i\leq d\}$ of $V_K$ such that for all $i$, we have that $\<{e_i}$ is the eigenspace of $\ell_i$, and $e_{i}^{\sigma} = e_{i+1}$ for every $i\in [1\ldots d-1]$, and $e_d^{\sigma}=e_1$. That is, we have $e_{i}^{\sigma} = e_{\res_{d}(i+1)}$, for all $i$.
\end{lemma}
\begin{proof}
By \cite[Theorem 2.14]{LidlNiederreiter}, since the characteristic polynomial of $s$ is irreducible of degree $d$ over $F$, the eigenvalues of $s$ in $V_K$ (which are precisely the roots in $K$ of the characteristic polynomial of $s$) are as asserted and the $\ell_i$ are distinct. Thus there are $d$ eigenspaces of dimension $1$ in $V_K$. Fix an eigenvector $e_1$ of $\ell_1$, such that the first nonzero entry of $e_1$ is $1$, and for each $i$ with $2\leq i\leq d$, set $e_i := e_1^{\sigma^{i-1}}$. Then for each $i$, since $s\in \GL(V_F)$ and is therefore fixed under the action of $\sigma$:
\[e_i^s = e_1^{\sigma^{i-1}s} = (e_1 s)^{\sigma^{i-1}} = (\ell_1 e_1)^{\sigma^{i-1}} = \ell_1^{q^{i-1}}e_i = \ell_i e_i,\]
and so $e_i$ is an $\ell_i$-eigenvector for $s$ as required.\\\\
Moreover, since $\ell_d^q = \lambda^{q^{d}} = \lambda = \ell_1$, we have that $e_d^{\sigma}\in\<{e_1}$, that is, $e_d^{\sigma}$ is a scalar multiple of $e_1$. Since $e_1$ has its first nonzero coordinate equal to $1$, and since the action of $\sigma$ fixes this coordinate, we have $e_d^{\sigma}=e_1$.
\end{proof}
\begin{corollary}\label{EigenspacesOfIrreducibleElementsCor}
Let $s\in \GL(d,q)$, and suppose that there exists a $d'$-dimensional $s$-invariant subspace $U$ of $V$, such that $s$ acts irreducibly on $U$. Then $s|_U$ has $d'$ eigenspaces of dimension $1$ in $U_K$, and there exists a basis $\mathscr{E} = \{e_{i} \mid 1\leq i\leq d'\}$ for $U_K$ such that $e_{i}^{\sigma} = e_{\res_{d'}(i+1)}$.
\end{corollary}
\begin{definition}\label{SpecialDefn}
Let $G$ be a classical group of rank $r$ as in one of the lines of Table \ref{SpecialTable}. Let $d'$ be as in the $4$th column of the corresponding line of Table \ref{SpecialTable}. Then an element $s\in G$ is a {\it special element} if $s$ is a $\ppd(d,q;d')$-element, and there exists an $s$-invariant decomposition $V=U\oplus U'$ such that $\dim U = d'$; and $o(s|_{U})$ is a multiple of the value in the $5$th column of the appropriate line of Table \ref{SpecialTable}; and if $d'<d$, then $o(s|_{U'})$ is equal to the value in the $6$th column of the appropriate line of Table \ref{SpecialTable}.
\end{definition}
\begin{remark}
We frequently refer to our procedure \texttt{Initialise} `searching for special elements', but this is not strictly true. Special elements, as in Definition \ref{SpecialDefn}, are merely a \emph{subset} of the elements that \texttt{Initialise} can use: in practice we may use elements with smaller order than the values in Table \ref{SpecialTable} (i.e. certain powers of special elements), but the proof that such elements are suitable is neither interesting nor illuminating, and adds no value to the analysis of our algorithms. In practice we may essentially `replace' the appearances of $q^{d'/2}+1$ in Table \ref{SpecialTable} with $\frac{q^{d'/2}+1}{\gcd(q+1, q^{d'/2}+1)}$.
\end{remark}
\begin{table}
\begin{center}%
\begin{tabular}{|c|c|c|c|c|c|c|c|c|}
\hline
\hline
Case & $d$ & $G$ & $d'$ & $o(s|_{U})$ & $o(s|_{U'})$ & Conditions  \\
\hline
\hline
$A_r$ & $r+1$ & $\SL(d,q)$ & $d$ & $\dfrac{q^d-1}{q-1}$ & -- & -- \\
\hline
$^{2\!\!}A_r$ & $r+1$ & $\SU(d,q)$ & $d$ & $\dfrac{{\sqrt{q}}^{d}+1}{\sqrt{q}+1}$ & -- & $d$ odd, $q$ square  \\
        &       &          &  $d-1$ & $\dfrac{{\sqrt{q}}^{d'}+1}{\sqrt{q}+1}$ & $1$ & $d$ even, $q$ square \\
\hline
$B_r$  & $2r+1$ &  $\SO(d,q)$ &  $d-1$ & ${q^{d'/2}+1}$ & $1$ & --  \\ 
\hline
$C_r$  & $2r$ &  $\Sp(d,q)$ &  $d$ & ${q^{d/2}+1}$ & -- & -- \\
\hline
$D_r$  & $2r$ &  $\SO^+(d,q)$ &  $d-2$ & ${q^{d'/2}+1}$ & $q+1$ & --  \\
\hline
$^{2\!}D_r$  & $2r$ &  $\SO^-(d,q)$ &  $d$ & ${q^{d/2}+1}$ & -- & -- \\
\hline
\hline
\end{tabular}
\caption{Properties of Special Elements of Classical Groups}\label{SpecialTable}
\end{center}
\end{table}
The subspace $U$ in Definition \ref{SpecialDefn} is uniquely determined by $s$, and $s$ acts irreducibly on $U$; if $d' < d$, then $s$ also acts irreducibly on $U'$ as a consequence of the condition on $o(s|_{U'})$. Our ultimate goal is a basis for $V_L$, where $L$ is an extension field of $F$ satisfying certain conditions: we define these conditions below.
\begin{definition}\label{Sigma Relations for V}
Let $G$ be a classical group over $F$, let $V$ be the natural $FG$-module, and let $\sigma$ be the Frobenius automorphism of $K/F$, where $K$ is an extension of $F$ of degree $d'$ as given in Table \ref{SpecialTable}. Let $\mathscr{F}:= \{f_i \mid 1\leq i\leq d\}$ be a basis for $V$: then we say $\mathscr{F}$ \emph{satisfies the almost-$\sigma$-relations for $V$} if the following hold:
\begin{enumerate}
 \item for $1\leq i \leq d'-1$, we have that $f_i^{\sigma} = f_{i+1}$, and $f_{d'}^{\sigma} \in \<{f_1}$;
 \item if $d'=d-1$, then $f_d^{\sigma} = f_d$; and
 \item if $d'=d-2$, then $f_{d-1}^{\sigma}=f_d$ and $f_{d}^{\sigma} = f_{d-1}$.
\end{enumerate}
If, in addition, we have $f_{d'}^{\sigma} = f_1$, then we say $\mathscr{F}$ \emph{satisfies the $\sigma$-relations for $V$}.
\end{definition}
We now describe explicitly the eigenstructure of a special element on $V_K$:
\begin{lemma}\label{SpecialEigLem}
Let $G$ be a classical group of rank $r$, let $d'$ be as in Table \ref{SpecialTable}, and let $s\in G$ be a special element. Let $K=\F{q^{d'}}$, and let $U,U'$ be as in Definition \ref{SpecialDefn}. Then the eigenvalues of $s$ in its action on $V_K$ are
\[
\{\ell_{i} \mid 1\leq i\leq d \},
\]
where
\[
\ell_i = \begin{cases} 
\lambda^{q^{i-1}} & \text{ for }1\leq i\leq d', \\
\mu^{q^{i-d'-1}} & \text{ for } d' < i \leq d, \\
\end{cases} 
\]
 for $\lambda,\mu\in K$ satisfying $\lambda = o(s|_{U}), \mu = o(s|_{U'})$ as in the $5$th and $6$th entry respectively in the appropriate line of Table \ref{SpecialTable}. Moreover, there exists a basis 
\[
\mathscr{E}:= \mathscr{E}(s,V):=\{e_i:=e_{s,V,i}\mid 1\leq i\leq d\}
\] 
such that $\mathscr{E}$ satisfies the $\sigma$-relations for $V$ as in Definition \ref{Sigma Relations for V}.
\end{lemma}
\begin{proof}
If $d'=d$ then the result follows immediately from \ref{EigenspacesOfIrreducibleElements}. If $d'=d-1$, then by definition we have $o(s|_{U'}) = 1$, and so $s$ acts trivially on $U'$: that is, $\ell_d=1$ is an eigenvalue of $s$, and the result follows from this fact and Corollary \ref{EigenspacesOfIrreducibleElementsCor}, since $s$ acts irreducibly on both $U$ and $U'$.\\\\
If $d'=d-2$, then $s|_{U'}\in \GL(2,q)$ has order $q+1$, and hence acts irreducibly on $U'$ (since all proper nontrivial subspaces of $U'$ are $1$-dimensional), and the result follows by applying Corollary \ref{EigenspacesOfIrreducibleElementsCor} separately to $U$ and $U'$.
\end{proof}
\section{Arithmetic}\label{Arithmetic Section}
In this section we prove a series of results in modular arithmetic which will be used in Section \ref{Eigenstructure on VotimesV Section} below. These results have been separated so that the later results, which are more relevant in the bigger picture, are not obfuscated by these long, repetitive and technical proofs.\\\\
By Lemma \ref{EigenstructureLemma}, the multiset of eigenvalues of a special element of a classical group $G$ in its action on the tensor product $(V\otimes_F V)_K$ is the multiset
\[
\Sigma := \{\ell_{ij}\mid i, j\in \{1,\ldots, d\} \},
\]
where each $\ell_{ij}$ is an element of $K=\F{q^{d'}}$, and is the product of a pair of eigenvalues of $s$ in its action on $V_K$ as described in Lemma \ref{EigenstructureSpecialVotimesV} below (note that the details of this are not required for the results in this section, except as motivation). This multiset has size $d^2$, but contains repeated values: in all cases, for example, we have $\ell_{ij}=\ell_{ji}$. In this section we provide necessary conditions for other coincidences to occur.
\begin{proposition}\label{Arithmetic Prop Easy}
Suppose $q$ is a prime power, that $d'\in\Z{}$ with $d'\geq 4$, and suppose there exists $j,m,n\in \Z{}$, with $1\leq j\leq d'/2$, $1\leq m<n\leq d'$ and satisfying
\begin{equation}\label{eigenvalue eqn Easy} 
1+q^{j-1} -q^{m-1}-q^{n-1} \equiv 0 \pmod{\frac{q^{d'}-1}{q-1}}. 
\end{equation}
Then $m = 1$ and $n=j$.
\end{proposition}
\begin{proof}
For all such $j,m,n$, we have
\[
1+ q^{j-1} - q^{m-1} - q^{n-1} 
\leq 1 + q^{d'/2-1} - 1 - 1 
= -1 + q^{d'/2} 
< \displaystyle\frac{q^{d'}-1}{q-1},
\]
and on the other hand,
\begin{align*}
1+ q^{j-1} - q^{m-1} - q^{n-1} 
&\geq 2 - q^{d'-1} - q^{d'-2} 
= -\displaystyle\frac{q^{d'} - q^{d'-2} + 2-2q}{q-1}
> -\displaystyle\frac{q^{d'} - 1}{q-1}
\end{align*}

and so we have equality in (\ref{eigenvalue eqn Easy}), not just equivalence modulo $\frac{q^{d'}-1}{q-1}$. Thus
\[
 1+q^{j-1} = q^{m-1}+q^{n-1}
\]
and reducing modulo $q$ we have that $m=1$, from which it immediately follows that $n=j$.
\end{proof}
We now address the `hard case', where $d'$ is even and $o(\lambda)$ is $q^{d'/2}+1$. 
We seek solutions to the equation
\begin{equation}\label{eigenvalueeqn} 
1+q^{j-1} +\epsilon_m q^{m'-1}+\epsilon_n q^{n'-1} = t({q^{{d'}/2}+1}) 
\end{equation}
for integer values of $j,m',n',\epsilon_m,\epsilon_n, t$ with $1\leq j, m',n' \leq d'/2$, $\epsilon_m, \epsilon_n\in\{-1,1\}$. We make an important distinction here: due to the fact that the Symmetric Square module contains the Alternating Square module when $q$ is even (and therefore we do not consider it in this paper), we \emph{need not consider} the case that $j=1$ when $q$ is even. For completeness (and for future use) we still consider the cases that apply to the Alternating Square (i.e. $q$ even and $j \neq 1$).
\begin{lemma}\label{CoincidencePrelim1}
Suppose $q$ is a prime power, that $d'\in\Z{}$ is even with $d'\geq 6$, and suppose that $j,m',n', \epsilon_m, \epsilon_n, t \in \Z{}$, with $1\leq j,m',n'\leq d'/2$,  $\epsilon_m, \epsilon_n \in \{\pm 1\}$, and $(m',\epsilon_m)\neq(n',\epsilon_n)$, satisfy (\ref{eigenvalueeqn}).\\\\
Then $t\in \{0,1\}$, and if $t=1$ then $q=2$, and up to switching $m,n$, we have that $j=\frac{d'}{2}-1, \epsilon_m= 1, m'=\frac{d'}{2}-1,\epsilon_n = 1, n' = \frac{d'}{2}$.
\end{lemma}
\begin{proof}
Since the pair $(m',\epsilon_m)\neq(n',\epsilon_n)$, we have 
\[
2- q^{{d'}/2-1} - q^{{d'}/2-2} \leq 1+q^{j-1} +\epsilon_m q^{m'-1}+\epsilon_n q^{n'-1} \leq 1+q^{{d'}/2-2} + 2q^{{d'}/2-1},
\]
and so
\[
\displaystyle\frac{2- q^{{d'}/2-1} - q^{{d'}/2-2}}{q^{{d'}/2}+1} \leq t \leq \displaystyle\frac{ 1+q^{{d'}/2-2} + 2q^{{d'}/2-1}}{q^{{d'}/2}+1}.
\]
It is readily checked that the upper bound is less than $1$ if $q \geq 3$, and less than $2$ if $q=2$, while the lower bound is greater than $-1$ for all $q$. Suppose that $t=1$: then $q=2$ (we continue to write $q$), and (\ref{eigenvalueeqn}) is
\[
1 + q^{j-1} + \epsilon_m q^{m'-1} + \epsilon_n q^{n'-1} = q^{d'/2} + 1.
\]
Now the largest value the left hand side can take is when $j=d'/2, \epsilon_m=\epsilon_n=1, m'=d'/2-1, n'=d'/2$, and in this case we have
\[
1+ q^{j-1} +\epsilon_m q^{m'-1}+\epsilon_n q^{n'-1} = 1+ 2q^{d'/2-1} + q^{d'/2-2} = 1+ q^{d'/2}+q^{d'/2-2} > 1+ q^{d'/2}.
\]
The next-largest value is attained when $j=d'/2-1, \epsilon_m=\epsilon_n=1, m'=d'/2-1, n'=d'/2$: in this case
\[
1+ q^{j-1} +\epsilon_m q^{m'-1}+\epsilon_n q^{n'-1} = 1 + 2q^{d'/2-2} + q^{d'/2-1} =  1 + 2q^{d'/2-1} = 1 + q^{d'/2}.
\]
This is precisely the solution given. All \emph{other} combinations of $j,\epsilon_m, m', \epsilon_n, n'$ give smaller values for the left hand side, and so cannot yield solutions.
\end{proof}
\begin{proposition}\label{Arithmetic Prop Main}
Suppose $q$ is a prime power, that $d'\in\Z{}$ is even with $d'\geq 6$, and suppose that $j,m',n', \epsilon_m, \epsilon_n \in \Z{}$, with $1\leq j,m',n'\leq d'/2$,  $\epsilon_m, \epsilon_n \in \{\pm 1\}$, and $(m',\epsilon_m)\neq(n',\epsilon_n)$, satisfy
\begin{equation}\label{eigenvalue eqn d modulo}
1+q^{j-1} +\epsilon_m q^{m'-1}+\epsilon_n q^{n'-1} \equiv 0 \pmod{q^{{d'}/2}+1}. 
\end{equation}
Then one of the following holds:
\begin{enumerate}
\item $\epsilon_m = \epsilon_n = -1$ and $\{1,j\} = \{m',n'\}$ (the \emph{trivial solution});
\item  $q=2, j=\frac{d'}{2}-1, \epsilon_m= 1, m'=\frac{d'}{2}-1,\epsilon_n = 1, n' = \frac{d'}{2}$.
\item $q=3, j=1, m'=\epsilon_m = 1, n'=2, \epsilon_n = -1$; or
\item $q=2, j=2, m'=\epsilon_m = 1, n'=3, \epsilon_n = -1$.
\end{enumerate}
\end{proposition}
\begin{proof}
Suppose that we are not in case (ii): then by Lemma \ref{CoincidencePrelim1}, we have equality in \eqref{eigenvalue eqn d modulo}. Reducing modulo $q$ we have
\begin{equation}\label{mod q temp}
1+q^{j-1} +\epsilon_m q^{m'-1}+\epsilon_n q^{n'-1} \equiv 0 \pmod{q}.
\end{equation}
Since $(m',\epsilon_m)\neq(n',\epsilon_n)$ implies $\epsilon_m q^{m'-1}\neq \epsilon_n q^{n'-1}$, they cannot both be 1 nor both $-1$, and so the left hand side of (\ref{mod q temp}), when reduced modulo $q$, is equal to $0,1,2$ or $3$. Since $q\geq 2$, only the values $0, 2$ and $3$ can possibly be equivalent to $0$ modulo $q$. We treat each case separately, and refer to the value of the left hand side of (\ref{mod q temp}) after it has been reduced modulo $q$ as the \emph{reduced left hand side} of (\ref{mod q temp}).\\\\
If the reduced left hand side of (\ref{mod q temp}) is $3$, then $q=3, j=1$ and exactly one of $\epsilon_m q^{m'-1}, \epsilon_n q^{n'-1} =1$, say $m'=\epsilon_m=1$ and $n' >1$. Then (\ref{eigenvalueeqn}) yields
\[
3+\epsilon_nq^{n'-1} = 0,
\]
forcing $\epsilon_n=-1, n'=2$: This is solution (ii).\\\\
If the reduced left hand side of (\ref{mod q temp}) is $2$, then $q=2$ and one of the terms in the left hand side is $1$, say $\epsilon_m q^{m'-1}=1$ (noting that when $q$ is even we have $j>1$). Then (\ref{eigenvalueeqn}) is
\[
2+ q^{j-1}+\epsilon_n q^{n'-1} = 0,
\]
forcing $\epsilon_n=-1$, and so 
\[
2+ q^{j-1} = q^{n'-1}.
\]
There is only one way in which `$2$ plus a power of $2$' can equal a power of $2$: namely $2+2=4$, and so $j=2, n'=3$. This is solution (iii).\\\\
Finally, if the reduced left hand side of (\ref{mod q temp}) is zero, then $j\geq 2$ and one of $\epsilon_m q^{m'-1}, \epsilon_n q^{n'-1} = -1$, say $m'=1,\epsilon_m=-1$ and $n' > 1$. Then (\ref{eigenvalueeqn}) reduces to
\[
q^{j-1} +\epsilon_n q^{n'-1} = 0,
\]
and so $\epsilon_n=-1, n'=j$. Thus $m=1, j=n, \epsilon_m=\epsilon_n=-1$, the trivial solution.
\end{proof}
%
%
%
%
%
%
%
%
%
%
%
%
%
%
%
%
%
%
%
%
%
%
%
%
%
In the case $G=\SU(d,q)$, we have that $q$ is a square, and $\lambda$ has smaller order than $q^{d'/2}+1$, and so we must treat it separately (though we use similar methods), and we must solve the following equation, which bears a strong similarity to \eqref{eigenvalueeqn}: note that in the Unitary case, we have $d'$ odd.
\begin{proposition}\label{Arithmetic Prop Unitary}
Suppose $q$ is a square prime power, that $d'\in\Z{}$ with $d'\geq 3$, and suppose that $j,m',n', \epsilon_m, \epsilon_n \in \Z{}$, with $1\leq j,m',n'\leq (d'-1)/2$,  $\epsilon_m, \epsilon_n \in \{\pm 1\}$, and $(m',\epsilon_m)\neq(n',\epsilon_n)$, satisfy
\begin{equation}\label{eigenvalue eqn unitary modulo}
1+q^{j-1} +\epsilon_m q^{m'-1}+\epsilon_n q^{n'-1} \equiv 0 \pmod{\frac{\sqrt{q}^{{d'}}+1}{\sqrt{q}+1}}
\end{equation}
Then $\epsilon_m = \epsilon_n = -1$ and $\{1,j\} = \{m',n'\}$.
\end{proposition}
\begin{proof}
Suppose that 
\[
1+q^{j-1} +\epsilon_m q^{m'-1}+\epsilon_n q^{n'-1} = t\left(\frac{\sqrt{q}^{{d'}}+1}{\sqrt{q}+1}\right).
\]
Now since the pair $(m',\epsilon_m)\neq(n',\epsilon_n)$, we have 
\[
2- q^{(d'-1)/2-1} - q^{(d'-1)/2-2} \leq 1+q^{j-1} +\epsilon_m q^{m'-1}+\epsilon_n q^{n'-1} \leq 1+q^{(d'-1)/2-2} + 2q^{(d'-1)/2-1},
\]
and so 
\[
\frac{q^{1/2} + 1}{q^{{d'}/2}+1} \left(2- q^{(d'-1)/2-1} - q^{(d'-1)/2-2}\right) \leq t \leq \frac{q^{1/2} + 1}{q^{{d'}/2}+1} \left( 1+q^{(d'-1)/2-2} + 2q^{(d'-1)/2-1} \right).
\]
Onve again it is simple to check (noting that $q \geq 4$) that the left hand side is greater than $-1$, while the right hand side is less than $1$. Thus $t=0$ and we have equality in \eqref{eigenvalue eqn unitary modulo}, and the result follows by the arguments in the proof of Proposition \ref{Arithmetic Prop Main}, noting that the exceptional cases with $q$ not a square do not arise.
\end{proof}
For the sake of brevity we state the remaining results of this Section without proof: the statements and analysis are similar to the results above, and the full proofs (as well as more detailed proofs of the results above) are available in Section 2.4 of \cite{corr2013estimation} (we will refer the reader to the specific results as we go).
\begin{proposition}[\cite{corr2013estimation}, Proposition 2.4.13]\label{Arithmetic Prop d-2}
Suppose $q$ is a prime power, that $d' \geq 8$, and suppose that $t,j,m', \epsilon_m \in \Z{}$, with $1\leq j,m'\leq d'/2$,  $\epsilon_m, \in \{\pm 1\}$, satisfy 
\begin{equation}\label{eigenvalue eqn d-2 modulo}
1+q^{j-1} + \epsilon_m q^{m'-1} \equiv 0 \pmod{q^{d'/2}+1}.
\end{equation}
Then $q=2$, and one of the following holds:
\begin{enumerate}
\item $d'=10, j=3, \epsilon_m=-1, m'= 5$ (that is, $3\times (1+4-16) = -32-1$);
\item $j=1, \epsilon_m=-1, m'=2$ (that is, $3\times (1+1-2) = 0$); or
\item $d'=10, \epsilon_m=1, \{j,m'\} = \{2,4\}$ (that is, $3\times ( 1+2+8) = 32+1$).
\end{enumerate}
\end{proposition}
Proposition \ref{Arithmetic Prop d-2} solves the issue of whether eigenvalues of the form $\ell_{ij},\ell_s m_t$ can be equal, in the case $d'=d-2$ below. In the case $d'=d-1$, we must compare eigenvalues of the form $\ell_{ij},\ell_m$:
\begin{corollary}[\cite{corr2013estimation}, Corollary 2.4.14]\label{Arithmetic Cor d-1}
Suppose $q$ is a prime power, that $d' \geq 8$, and suppose that $j,m', \epsilon_m \in \Z{}$, with $1\leq j,m'\leq d'/2$,  $\epsilon_m \in \{\pm 1\}$, satisfy 
\begin{equation}\label{eigenvalue eqn d-1 modulo}
1+q^{j-1} + \epsilon_m q^{m'-1} \equiv 0 \pmod{q^{d'/2}+1}.
\end{equation}
Then $q=2, j=1, \epsilon_m=-1, m'=2$.
\end{corollary}
\begin{proof}
Multiplying by $(q+1)$ implies that (\ref{eigenvalue 	eqn d-2 modulo}) holds, and so the result immediately follows by testing the solutions found in Proposition \ref{Arithmetic Prop d-2}: of these, only (ii) satisfies \eqref{eigenvalue eqn d-1 modulo}.
\end{proof}
Once again we must treat the Unitary case separately:
\begin{corollary}[\cite{corr2013estimation}, Corollary 2.4.15]\label{Arithmetic Cor d-1 Unitary}
Suppose $q$ is a square prime power, that $d' \geq 5$. Then there is no $j,m', \epsilon_m \in \Z{}$, with $1\leq j,m'\leq (d'-1)/2$,  $\epsilon_m \in \{\pm 1\}$, satisfying
\begin{equation}\label{eigenvalue eqn d-1 modulo Unitary}
1+q^{j-1} + \epsilon_m q^{m'-1} \equiv 0 \pmod{\frac{\sqrt{q}^{d'}+1}{\sqrt{q}+1}}.
\end{equation}
\end{corollary}
\section{The Eigenstructure of Special Elements on $(V\otimes V)_K$}\label{Eigenstructure on VotimesV Section}
In this section we determine the precise eigenstructure of a special element $s\in G$ in its action on $(V\otimes V)_K$, which will enable us to determine the eigenstructure of $s$ in its action on $S^2(V)_K$. This eigenstructure is the crux of the procedures \texttt{Initialise} and \texttt{FindPreimage}. Throughout this section, define $\res_{d'}(i)$ as the unique integer $j$ such that $1\leq j\leq d'$ and $j\equiv i\pmod{d'}$.
\subsection{Coincident Eigenvalues $\ell_{ij}$}
There are two ways in which eigenvalues $\ell_{ij}$ may coincide: there are cases where $\ell_{ij}=1$ (leading to a nontrivial fixed-point space of $s$), or where two eigenvalues are not $1$, but coincide anyway.
\begin{lemma}\label{Ones}
Let $\lambda\in K$, let $\ell_{i}=\lambda^{q^{i-1}}$ for $1\leq i \leq {d'}$, and let $\ell_{ij}=\ell_i\ell_j$ for $1\leq i\leq j\leq {d'}$. Suppose that the order of $\lambda$ is divisible by a primitive prime divisor $r$ of $q^{d'}-1$, and suppose for some integer $t$, not divisible by $r$, we have $\ell_{ij}^t=1$. Then ${d'}$ is even, and $j-i = {d'}/2$.
\end{lemma}
\begin{proof}
If $\ell_{ij}^t = 1$ then $1= \lambda^{t(q^{i-1} + q^{j-1})}=\lambda^{tq^{i-1}(1+q^{j-i})}$, so $r$ divides $tq^{i-1}(1+q^{j-i})$. Since $r$ does not divide $q$ or $t$, and $r$ is prime, it follows that $r\mid (1+q^{j-i})$, and so $r$ divides $q^{2(j-i)} -1$. Since $r$ is a primitive prime divisor of $q^{d'}-1$, it follows that $d'\mid 2(j-i)$, and since $0 <j-i  < d'$, we have $0 < 2(j-i) < 2d'$ and so $2(j-i)=d'$. 
\end{proof}
Lemma \ref{Ones}, with $t=1$, is crucial in determining when a special element $s$ has an eigenvalue $1$. Note that Lemma \ref{Ones} provides only a necessary condition, and not a sufficient condition: in some cases, the eigenvalue $\ell_{1,{d'}/2+1}$ may be different from 1 (this is dependent on the order of $\lambda$).\\\\
The existence of a fixed-point space of $s$ in its action on $(V\otimes V)_K$ may seem unfortunate (in the sense that it guarantees that not all of the eigenspaces can be $1$-dimensional). However, in Section \ref{sec: Modules and Reps} we observe that, in all but the Unitary and Linear cases, $G$ has fixed points in its action on $\M(V) \cong V\otimes V$, and so these products equalling $1$ is inevitable -- we cannot hope to find an element in $G$ with no fixed points.\\\\
We now address coincidences among the $\ell_{ij}$ other than those corresponding to fixed points: we seek pairs $(i,j), (m,n)$ such that $\ell_{ij}=\ell_{mn}$. We begin by exploiting the symmetry of the problem under the action of $\sigma$ as much as we can.
\begin{lemma}\label{CoincidencePrelim}
Suppose $q$ is a prime power, and $d'$ an even integer with $d'\geq 6$. Let $K=\F{q^{d'}}$, let $\lambda\in K$, and suppose that $o(\lambda)$ is divisible by a primitive prime divisor $r$ of $q^{d'}-1$. Let $\ell_{i} = \lambda^{q^{i-1}}$, for $1\leq i\leq d'$, and let $\ell_{ij}=\ell_i\ell_j$ for $1\leq i, j\leq d'$. Suppose that there exist integers $i,j,m,n\in \{1,\ldots, d'\}$, satisfying 
\[
\ell_{ij}=\ell_{mn}.
\]
Then at least one of the following holds:
\begin{enumerate}
\item $\{i,j\} = \{m,n\}$;
\item $d'$ is even, and $\res_{d'}(j-i) =\res_{d'}(n-m) = d'/2$;
\item There exist integers $k,s,t,\alpha$, with $1\leq k\leq d'/2$, and $1\leq s<t\leq d'$, such that $\ell_{1k}=\ell_{st}$, and $\ell_{ij} = \ell_{1k}^{q^{\alpha}}$.
\end{enumerate}
\end{lemma}
\begin{proof}
Suppose that $\res_{d'}(j-i)=\res_{d'}(n-m)$. Then setting $t'= \res_{d'}(i-m)$, we have $\res_{d'}(m+t') = i$, and $\res_{d'}(n+t') = \res_{d'}(n-m+i) = \res_{d'}(j-i+i)=j$, and so 
\[
\ell_{mn}^{q^{t'}} = \ell_{m+t',n+t'} = \ell_{ij}=\ell_{mn}.
\]
Thus $\ell_{mn}^{q^{t'}-1} = 1$. If $t'=d'$ then $m=i, j=n$ and we are in case (i). Assume $t' <d'$: then since $r$ is a primitive prime divisor of $q^{d'}-1$, $r$ does not divide $q^{t'}-1$. Then by Lemma \ref{Ones} (with $t=q^{t'}-1$), we have $\res_{d'}(n-m) = \res_{d'}(j-i) = d'/2$.\\\\
Suppose, then, that $\res_{d'}(j-i) \neq \res_{d'}(n-m)$ and so at least one of $\res_{d'}(j-i) , \res_{d'}(n-m)$ is distinct from $d'/2$, and at least one is distinct from $d'$. If $\res_{d'}(j-i)= d'/2$, or if $\res_{d'}(n-m) = d'$ then we switch $\{i,j\}$ with $\{m,n\}$: then we may assume that $\res_{d'}(n-m)\neq d'$ and $\res_{d'}(j-i) \neq d'/2$. \\\\
If $\res_{d'}(j-i) < d'/2$, then setting $k = \res_{d'}(j-i+1)$, and $\alpha = \res_{d'}(i-1)$, we have $\ell_{ij}^{q^{d'-\alpha}} =\ell_{ij}^{q^{-i+1}} = \ell_{1,j-i+1} = \ell_{1k}$. Set $\{s,t\}= \{\res_{d'}(m-i+1), \res_{d'}(n-i+1)\}$, with $s,t$ ordered so that $s < t$. Note that since $\res_{d'}(n-m) \neq d'$ we have $m\neq n$, implying that $s\neq t$. Then $\ell_{ij}=\ell_{1k}^{q^{\alpha}}$,
\[
\ell_{1k}= \ell_{mn}^{q^{d'-\alpha}} = \ell_{mn}^{q^{-i+1}} = \ell_{m-i+1,n-i+1} = \ell_{st},
\]  
and $k = \res_{d'}(j-i+1) < d'/2+1$, and so $k\leq d'/2$ as required.\\\\
On the other hand, if $\res_{d'}(j-i)>d'/2$, then $\res_{d'}(i-j) < d'/2$. Then setting $k= \res_{d'}(i-j)+1$, and $\alpha = \res_{d'}(j-1)$, we have $\ell_{ij}^{q^{d'-\alpha}} = \ell_{ij}^{q^{-j+1}}=\ell_{i-j+1,1} = \ell_{1k}$. Set $\{s,t\}= \{\res_{d'}(m-j+1), \res_{d'}(n-j+1)\}$ with $s,t$ again ordered so that $s< t$. Then we have $\ell_{ij}=\ell_{1k}^{q^{\alpha}}$, $\ell_{1k} = \ell_{st}$, and again $k \leq d'/2$ as required.
\end{proof}
The upshot of Lemma \ref{CoincidencePrelim} is that in our search for coincidences $\ell_{ij}=\ell_{mn}$ among our eigenvalues, we may assume without loss of generality that $i=1, j \leq d'/2$ and $m\neq n$. The first case we deal with is the Linear Case, where the order of $\lambda$ is largest:
\begin{lemma}\label{coincidence SL}
Suppose $q$ is a prime power, and $d'$ an integer with $d'\geq 4$. Let $\lambda\in K$ have order a multiple of $\frac{q^{{d'}}-1}{q-1}$. Let $\ell_{i} = \lambda^{q^{i-1}}$, for $1\leq i\leq d'$, and let $\ell_{ij}=\ell_i\ell_j$ for $1\leq i\leq j\leq d'$. Suppose there exist integers $j,m,n$ such that $\ell_{1j}=\ell_{mn}$, with $1\leq j\leq d'/2$ and $1\leq m< n\leq d'$, with $j\neq 1$ if $q$ is even. Then $(1,j) = (m,n)$.
\end{lemma}
\begin{proof}
If $\ell_{1j}=\ell_{mn}$ then $\lambda^{1+q^{j-1} -q^{m-1} - q^{n-1}} = 1$, and hence 
\[
1+ q^{j-1} - q^{m-1} - q^{n-1} \equiv 0 \pmod{\frac{q^d-1}{q-1}}.
\]
This is a solution of equation \eqref{eigenvalue eqn Easy}, and the result then follows from Proposition \ref{Arithmetic Prop Easy}.
\end{proof}
In all other cases things are more difficult, and most of Section \ref{Arithmetic Section} is devoted to aspects of the proof that the $\ell_{ij}$ rarely coincide. 
\begin{lemma}\label{coincidence}
Suppose $q$ is a prime power, and $d'$ an integer with $d'\geq 6$. Let $\lambda\in K$ have order ${q^{{d'}/2}+1}$. Let $\ell_{i} = \lambda^{q^{i-1}}$, for $1\leq i\leq d'$, and let $\ell_{ij}=\ell_i\ell_j$ for $1\leq i\leq j\leq d'$. Suppose there exist integers $j,m,n$ such that $\ell_{1j}=\ell_{mn}$, with $1\leq j\leq d'/2$ and $1\leq m< n\leq d'$, with $j\neq 1$ if $q$ is even. Then one of the following holds:
\begin{enumerate}
\item $\{1,j\}=\{m,n\}$ (the \emph{trivial coincidence});
\item $q=2, j=d'/2-1, m=d'-1, n=d$; 
\item $q=2$, $j=2$, $m=d'/2$, and $n=3$;
\item $q=3$, $j=1$, $m=2$, and $n=d'/2+1$.
\end{enumerate}
\end{lemma}
\begin{proof}
Solutions to $\ell_{1j}=\ell_{mn}$ are integer solutions to the equation
\[
1+q^{j-1} -q^{m-1}-q^{n-1} \equiv 0 \pmod{o(\lambda)},
\]
for $1\leq j\leq d'/2$, and $1\leq m \leq n\leq d'$. Now if $m > {d'}/2$, then
\[
q^{m-1} = q^{m-{d'}/2-1}(q^{{d'}/2} +1) -q^{m-{d'}/2-1},
\]
and so
\[
q^{m-1} \equiv -q^{m-{d'}/2-1} \pmod{q^{{d'}/2}+1}.
\]
The same argument holds if $n>{d'}/2$, and so we have
\[
1+q^{j-1} +\epsilon_m q^{m'-1}+\epsilon_n q^{n'-1} \equiv 0 \pmod{q^{{d'}/2}+1},
\]
where 
\[
m'=
\begin{cases} 
m &\text{ when } m \leq d'/2 \\ 
m-d'/2 &\text{ when } m > d'/2,
\end{cases} 
\quad \text{ and }\quad
\epsilon_m=
\begin{cases} 
-1 & \text{ when } m \leq d'/2 \\ 
+1 & \text{ when } m > d'/2,
\end{cases}
\]
and $n', \epsilon_n$ are defined likewise. Note that while $m'$ may be equal to $n'$, since $m < n$, we have $(\epsilon_m, m') \neq (\epsilon_n, n')$, and all of $j,m',n'$ lie between $1$ and $d'/2$. That is, $(j,m',\epsilon_m, n', \epsilon_n)$ is a set of solutions to equation
\[
1+ q^{j-1} +\epsilon_m q^{m'-1} + \epsilon_n q^{n'-1} \equiv 0 \pmod{q^{d'/2}+1}.
\]
This is precisely \eqref{eigenvalue eqn d modulo} in Section \ref{Arithmetic Section}, and the result follows by Proposition \ref{Arithmetic Prop Main}.
\end{proof}
Note here that the solutions (ii), (iii) in Lemma \ref{coincidence} are essentially `the same' coincidence: one can be obtained from the other by switching $(1,j)$ with $(m,n)$ and cycling under the action of $\sigma$.
%
%
%
%
%
%
%
%
We now address the Unitary case: note here that $d'$ is always odd (see Table \ref{SpecialTable}), and so $d'/2$ is not an integer. Thus when Lemma \ref{CoincidencePrelim} allows us to assume that $j \leq d'/2$, we may strengthen this to assume that $j\leq (d'-1)/2$.
\begin{lemma}\label{coincidence SU Odd}
Suppose $q$ is square a prime power, and $d'$ an odd integer with $d'\geq 3$. Let $\lambda\in K$ have order $\frac{\sqrt{q}^{{d'}}+1}{\sqrt{q}+1}$. Let $\ell_{i} = \lambda^{q^{i-1}}$, for $1\leq i\leq d'$, and let $\ell_{ij}=\ell_i\ell_j$ for $1\leq i\leq j\leq d'$. Suppose there exist integers $j,m,n$ such that $\ell_{1j}=\ell_{mn}$, with $1\leq j\leq (d'-1)/2$ and $1\leq m< n\leq d'$, with $j\neq 1$ if $q$ is even. Then $(1,j) = (m,n)$.
\end{lemma}
\begin{proof}
Define $\epsilon_m, m', \epsilon_n, n'$ as in the proof of Lemma \ref{coincidence} above: then by an identical argument, since $o(\lambda)\mid (q^{d'/2}+1)$ (where $q^{d'/2}$ denotes $\sqrt{q}^{d'}$), we have that $q^{m-1} \equiv \epsilon_m q^{m'-1}$ modulo $o(\lambda)$, and so a solution to $\ell_{1j} = \ell_{mn}$ corresponds to a solution to 
\[
1+ q^{j-1} +\epsilon_m q^{m'-1} + \epsilon_n q^{n'-1} \equiv 0 \pmod{\frac{\sqrt{q}^{{d'}}+1}{\sqrt{q}+1}}.
\]
This is precisely equation \eqref{eigenvalue eqn unitary modulo} in Proposition \ref{Arithmetic Prop Unitary}, and the result follows.
\end{proof}
 We now address the possibility of coincidence which are specific to the cases $d'=d-2,d-1$.
\begin{lemma}\label{CoincidencePrelim d-2}
Suppose $q$ is a prime power, and $d'$ an even integer with $d'/2\geq 3$. Let $K=\F{q^{d'}}$, let $\lambda\in K$, and let $\mu\in K$ have order $q+1$. Let $\ell_{i} = \lambda^{q^{i-1}}$, for $1\leq i\leq d'$, let $m_1=\mu, m_2=\mu^q$, and let $\ell_{ij}=\ell_i\ell_j$ for $1\leq i, j\leq d'$. Suppose that there exist integers $i,j,k\in \{1,\ldots, d'\}$, and $t\in \{1,2\}$, satisfying 
\[
\ell_{ij}=\ell_{k}m_t.
\]
Then there exist integers $r,s, n,\alpha$, with $1\leq r\leq d'/2, 1\leq s\leq d', 1\leq n\leq 2$, such that $\ell_{1r}=\ell_{s}m_n$, and $\ell_{ij} = \ell_{1r}^{q^{\alpha}}$.
\end{lemma}
\begin{proof}
Set $r = \min \{ \res_{d'}(j-i+1), \res_{d'}(i-j+1) \}$: if $r = \res_{d'}(j-i+1)$, then we have $\ell_{1k}^{q^{i-1}} = \ell_{1+i-1, k+i-1} = \ell_{ij}$, and so setting $\alpha=d'-i+1$, we have that $\ell_{1k}=\ell_{ij}^{q^{\alpha}}$, and 
\[
\ell_{1k} = \ell_{ij}^{q^{\alpha}} = (\ell_k m_t)^{q^{\alpha}} = \ell_{\res_{d'}(k+\alpha)} m_{\res_2(t+\alpha)}.
\]
Then setting $s=\res_{d'}(k+\alpha), n=\res_2(t+\alpha)$, the result holds. When $r = \res_{d'}(i-j+1)$, the result holds by an identical argument, with $\alpha = d'-j+1$.
\end{proof}
\begin{lemma}\label{Coincidence Special d-2}
Suppose $q$ is a prime power, and $d'$ an even integer with $d'/2\geq 3$. Let $\lambda\in K$ have order ${q^{{d'}/2}+1}$, and let $\mu \in K$ have order $q+1$. Let $\ell_{i} = \lambda^{q^{i-1}}$, for $1\leq i\leq d'$, let $m_i = \mu^{q^{i-1}}$ for $1\leq i\leq 2$, and let $\ell_{ij}=\ell_i\ell_j$ for $1\leq i\leq j\leq d'$. Suppose there exist integers $j, s,t$ such that $\ell_{1j}=\ell_{s}m_t$, with $1\leq j,s\leq d'$ and $1\leq t\leq 2$, with $j\neq 1$ if $q$ is even. Then $q=2$, and one of the following holds:
\begin{enumerate}
\item $d'=10, j=3, \epsilon_m=-1, m'= 5$;
\item $j=1, \epsilon_m=-1, m'=2$; or
\item $d'=10, \epsilon_m=1, \{j,m'\} = \{2,4\}$.
\end{enumerate}
\end{lemma}
\begin{proof}
Since $o(\mu) = q+1$, we have that $\ell_{1j}^{q+1} = (\ell_{s}m_t)^{q+1} = \ell_s^{q+1}$, and so $\lambda^{(1+q^{j-1})(q+1)} = \lambda^{q^{s-1}(q+1)}$: that is,
\[
(1+q^{j-1} - q^{s-1})(q+1) \equiv 0 \pmod{q^{d'/2}+1}.
\]
As in the proof of Lemma \ref{coincidence}, set 
\[
s'=
\begin{cases} 
s &\text{ when } s \leq d'/2 \\ 
s-d'/2 &\text{ when } s > d'/2,
\end{cases} 
\quad \text{ and }\quad
\epsilon_s=
\begin{cases} 
-1 & \text{ when } s \leq d'/2 \\ 
+1 & \text{ when } s > d'/2,
\end{cases}
\]
and we have that $\epsilon_s q^{s'-1} \equiv -q^{s-1}$ modulo $q^{d'/2}+1$, implying
\[
(1+q^{j-1} + \epsilon_s q^{s'-1})(1+q) \equiv 0 \pmod{q^{d'/2}+1}.
\]
This is precisely equation \eqref{eigenvalue eqn d-2 modulo}, and the result follows by Proposition \ref{Arithmetic Prop d-2}.
\end{proof}
\begin{lemma}\label{CoincidencePrelim d-1}
Suppose $q$ is a prime power, and $d'$ an even integer with $d'/2\geq 3$. Let $K=\F{q^{d'}}$, and let $\lambda\in K$. Let $\ell_{i} = \lambda^{q^{i-1}}$, for $1\leq i\leq d'$, let $m_1=\mu, m_2=\mu^q$, and let $\ell_{ij}=\ell_i\ell_j$ for $1\leq i, j\leq d'$. Suppose that there exist integers $i,j,k\in \{1,\ldots, d'\}$, satisfying 
\[
\ell_{ij}=\ell_{k}.
\]
Then there exist integers $r,s, \alpha$, with $1\leq r\leq d'/2, 1\leq s\leq d'$, such that $\ell_{1r}=\ell_{s}$, and $\ell_{ij} = \ell_{1r}^{q^{\alpha}}$.
\end{lemma}
\begin{proof}
This follows immediately by the same proof as Lemma \ref{CoincidencePrelim d-2}, replacing $\mu$ with $1$.
\end{proof}
\begin{lemma}\label{Coincidence Special d-1}
Suppose $q$ is a prime power, and that $d'$ is an even integer with $d'\geq 6$. Suppose that $\lambda\in K$ has order ${q^{{d'}/2}+1}$, let $\ell_{i} = \lambda^{q^{i-1}}$ for $1\leq i\leq d'$, and let $\ell_{ij}= \ell_{i}\ell_j$.\\\\
Suppose there exist integers $j, s,t$ such that $\ell_{1j}=\ell_{s}$, with $1\leq j \leq d'/2,1\leq s\leq d'$, with $j\neq 1$ if $q$ is even. Then $q=2, j=1, \epsilon_m=-1, m'=2$.
\end{lemma}
\begin{proof}
By an identical argument to the proof of Lemma \ref{Coincidence Special d-2} above (without raising to the $(q+1)$st power), we have that
\[
1+q^{j-1} + \epsilon_s q^{s'-1} \equiv 0 \pmod{q^{d'/2}+1}
\]
where $s', \epsilon_s$ are as defined in the proof of Lemma \ref{Coincidence Special d-2}. Then $j,\epsilon_s, s', q,d'$ are solutions to the equation \eqref{eigenvalue eqn d-1 modulo} in Section \ref{Arithmetic Section}. Then the result follows from Corollary \ref{Arithmetic Cor d-1}.
\end{proof}
Once again we must treat the Unitary case separately.
\begin{lemma}\label{Coincidence Special d-1 Unitary}
Suppose $q$ is a square prime power, and that $d'$ is an odd integer with $d'\geq 3$. Suppose that $\lambda\in K$ has order a multiple of $\frac{\sqrt{q}^{d'}+1}{\sqrt{q}+1}$, let $\ell_{i} = \lambda^{q^{i-1}}$ for $1\leq i\leq d'$, and let $\ell_{ij}= \ell_{i}\ell_j$.\\\\
Then there do not exist integers $j, s,t$ such that $\ell_{1j}=\ell_{s}$, with $1\leq j \leq (d'-1)/2,1\leq s\leq d'$.
\end{lemma}
\begin{proof}
Again by an identical argument to the proof of Lemma \ref{Coincidence Special d-2}, we have that
\[
1+q^{j-1} + \epsilon_s q^{s'-1} \equiv 0 \pmod{\frac{\sqrt{q}^{d'}+1}{\sqrt{q}+1}}
\]
where $s', \epsilon_s$ are as defined in the proof of Lemma \ref{Coincidence Special d-2}. Then $j,\epsilon_s, s', q,d'$ are solutions to the equation (\ref{eigenvalue eqn d-1 modulo Unitary}) in Section \ref{Arithmetic Section}. Then the result follows from Corollary \ref{Arithmetic Cor d-1 Unitary}.
\end{proof}
\begin{lemma}\label{EigenstructureSpecialVotimesV}
Let $G$ be a classical group as in one of the lines of Table \ref{SpecialTable}, and let $s\in G$ be a special element as defined in Definition \ref{SpecialDefn}, with eigenvalues $\{\ell_{i}\mid 1\leq i\leq d\}$ as in Lemma \ref{SpecialEigLem}, and let $\mathscr{E}(s,V)=\{e_i\mid 1\leq i\leq d\}$ be as defined in Lemma \ref{SpecialEigLem}. Suppose that in the Linear and Unitary cases, we have $d \geq 4$; in the remaining cases with $d'=d$ we have $d'\geq 6$; in the cases $d'=d-1,d-2$ we have $d'\geq 8$. Then the eigenvalues of $s$ in its action on $(V\otimes V)_K$ are
\[
\{\ell_{ij}:=\ell_i\ell_j\mid 1\leq i\leq j\leq d\},
\]
and the following hold:
\begin{enumerate}
\item if $\ell_{ij}=1$ then either $d'=d-1$ and $i=j=d$; or $d'=d-2$ and $\{i,j\} = \{d-1,d\}$; or $i \leq d',j \leq d'$, $\res_{d'}(j-i)=d'/2$, and $G$ is Symplectic or Orthogonal;
\item for each pair $(i,j)$ with $1\leq i \leq j\leq d$, the $\ell_{ij}$-eigenspace of $s$ in its action on $(V\otimes V)_K$ contains the tensor products $e_i\otimes e_j, e_j \otimes e_i$; and
\item the $\ell_{ij}$-eigenspace of $s$ in its action on $(V\otimes V)_K$ is precisely $\<{e_i\otimes e_j, e_j\otimes e_i}$, except when $\ell_{ij}=1$ and $G\in \{\Sp(d,q), \SO^{\epsilon}(d,q)\}$; or when $G \in \{\Sp(d,3), \SO^{\epsilon}(d,3)\}$, and $\res_{d'}(j-i)\in \{d', d'/2+1 \}$; or when $G\in \{\Sp(d,2),\SO^{\epsilon}(d,2)\}$ and $\res_{d'}(j-i)\in\{1,d'/2-2\}$.
\end{enumerate}
\end{lemma}
\begin{proof}
By Lemma \ref{Ones}, if $\ell_{ij} = 1$ and $1\leq i\leq j \leq d'$ then $d'$ is even and $\res_{d'}(j-i) = d'/2$. Since (when $d'=d-2$) for $\mu\in K$ of order $q+1$ we have that $\mu^2 \neq 1$, the only other possible pairs giving $\ell_{ij}=1$ are those listed. In the Linear case we have that $o(s) > \frac{q^d-1}{q-1}$, and so $\ell_{1,d'/2+1} = \lambda^{q^{d'/2} + 1} \neq 1$, and so $1$ is not an eigenvalue. In the Unitary case, we have that $d'$ is odd, and so $\ell_{ij}=1$ only if $i=j=d$, and so in this case the eigenspace of $1$ is precisely $\<{e_d \otimes e_d}$. In all other cases when $\ell_{ij}=1$, the eigenspace of $1$ has dimension greater than $1$. Thus (i) is proved. (ii) follows from Lemma \ref{EigenstructureLemma}; and (iii) follows from (i), and from Lemma \ref{CoincidencePrelim} and Lemma \ref{coincidence SL} in the Linear case; Lemma \ref{CoincidencePrelim} and Lemmas \ref{coincidence}, \ref{coincidence SU Odd} in non-Linear cases with $d'=d$; Lemma \ref{CoincidencePrelim d-1} and Lemmas \ref{Coincidence Special d-1}, \ref{Coincidence Special d-1 Unitary} when $d'=d-1$; and Lemma \ref{CoincidencePrelim d-2} and Lemma \ref{Coincidence Special d-2} when $d'=d-2$.
\end{proof}
Having completely described the action of a special element on $(V\otimes V)_K$, we turn to the submodule $S^2(V)$. The eigenstructure of a special element's action on $S^2(V)_K$ and can be `read' directly from Lemma \ref{EigenstructureSpecialVotimesV} using Lemma \ref{EigenstructureLemma}: in the next section, we provide concrete links between the action of an arbitrary $g\in G$ on various bases for $S^2(V)_K$ (and hence its irreducible constituents).
\section{The Action of a Special Element on the Symmetric Square $S^2(V)_K$}
\begin{lemma}\label{EigenstructureOnWKSymSquare}
Let $G$ be a Classical Group as in one of the lines of Table \ref{SpecialTable}, and let $s\in G$ be a special element as defined in Definition \ref{SpecialDefn}, with eigenvalues $\{\ell_{i}\mid 1\leq i\leq d\}$ as in Lemma \ref{SpecialEigLem}, and let $\mathscr{E}(s,V)=\{e_i\mid 1\leq i\leq d\}$ be as defined in Lemma \ref{SpecialEigLem}.\\\\
Define
\[
\mathscr{E}(s,S^2(V)) = \{e_{s,S^2(V), ij}\mid 1\leq i\leq j\leq d\},
\]
where 
\[
e_{s,S^2(V), ij}=e_i\otimes e_j + (1-\delta_{ij}) e_j\otimes e_i.
\]
Suppose that in the Linear and Unitary cases, we have $d \geq 4$; in the remaining cases with $d'=d$ we have $d'\geq 6$; and in the remaining cases with $d'=d-1,d-2$ we have $d'\geq 8$. Then the eigenvalues of $s$ in its action on $S^2(V)_K$ are
\[
\{\ell_{ij}:=\ell_i\ell_j\mid 1\leq i\leq j\leq d\},
\]
and the following hold:
\begin{enumerate}
\item if $\ell_{ij}=1$ then $i \leq d',j \leq d'$ and the condition in the $5$th column of the appropriate line of Table \ref{SpecialTable} holds;
\item for each pair $(i,j)$, the $\ell_{ij}$-eigenspace of $s$ contains $e_{s,S^2(V), ij}$; and
\item the $\ell_{ij}$-eigenspace of $s$ is precisely $\<{e_{s,S^2(V), ij}}$, except when either $\ell_{ij}=1$ and $G\in \{\Sp(d,q), \SO^{\epsilon}(d,q)\}$; or $G \in \{\Sp(d,3), \SO^{\epsilon}(d,3)\}$.
\end{enumerate}
\end{lemma}
\begin{proof}
This follows immediately from Lemma \ref{EigenstructureSpecialVotimesV} and \ref{EigenstructureLemma} (note that the exceptional cases in Lemma \ref{EigenstructureSpecialVotimesV} for $q=2$ do not apply here as $q$ is odd).
\end{proof}
\begin{lemma}\label{AppLem SymSquare}
Let $G$ be a classical group, let $W$ be an $FG$-module isomorphic to $S^2(V)$, and let $\App_W$ be the set of pairs $(i,j)$ of integers such that, for any special element $s\in G$, the $\ell_{ij}$-eigenspace of $s$ in its action on $W$ has dimension $1$. Then:
\begin{enumerate}
\item if $G\in \{\SL(d,q), \SU(d,q)\}$, then $d'=d$ and $\App_W = \{ (i,j) \mid 1\leq i \leq j\leq d\}$;
\item if $G\in \{ \Sp(d,q), \SO^{-}(d,q)\}, q \geq 5$, then $d'=d$ and $(i,j) \in \App_W$ if and only if $1\leq i\leq j\leq d$, and $j-i\neq d/2$;
\item if $G\in \{ \SO^{\circ}(d,q) \text{ for $d$ odd}\}, q \geq 5$, then $d'=d-1$ and $(i,j)\in\App_W$ if and only if either $1\leq i\leq j\leq d'$, and $j-i\neq d'/2$; or $1\leq i\leq d'$ and $j=d$;
\item if $G = \SO^+(d,q), q \geq 5$, then $d'=d-2$ and $(i,j)\in\App_W$ if and only if either $1\leq i\leq j\leq d'$, and $j-i\neq d'/2$; or $1\leq i\leq d'$ and $d'< j \leq d$; or $d'< i\leq d$ and $j=i$.
\end{enumerate}
Note that in all cases, if $1\leq i\leq j\leq d', j-i \neq d'/2$, we have $(i,j)\in \App_W$.
\end{lemma}
\begin{proof}
This follows immediately from Lemma \ref{EigenstructureOnWKSymSquare}. Note that while the eigenvalues and eigenvectors depend upon $s$, the set of pairs $(i,j)\in\App_W$ does not.
\end{proof}
While the notation $e_{s,S^2(V),ij}$ is cumbersome, there are very many modules and bases to consider, and it is sometimes needed for clarity. We will often simply write $e_{ij}$ when there is no ambiguity.
\begin{definition}\label{Sigma Relations SymSquare}
Let $G$ be a classical group as in one of the lines of Table \ref{SpecialTable}, let $s\in G$ be a special element as in Definition \ref{SpecialDefn}, and for $1\leq i \leq j\leq d$, let $e_{ij}:=e_{s,S^2(V), ij}$ be as defined in Lemma \ref{EigenstructureOnWKSymSquare}. Let $\sigma$ be the Frobenius automorphism of the extension $K/F$.\\\\
Suppose that $\mathscr{F} = \{f_{ij} \mid 1\leq i\leq j\leq d\}$ is a basis for $S^2(V)_K$. Then $\mathscr{F}$ is said to satisfy the \emph{$\sigma$-relations for $S^2(V)$} if the following hold:
\begin{enumerate}
\item for $1\leq i\leq j\leq d'-1$, $f_{ij}^{\sigma} = f_{i+1,j+1}$;
\item for $1\leq i\leq d'-1$, $f_{id'}^{\sigma} = f_{1,i+1}$, and $f_{d'd'}^{\sigma} = f_{11}$;
\item if $d'=d-1$ then, for $1\leq i \leq d'$, $f_{i, d}^{\sigma} = f_{i+1,d}$, and $f_{dd}^\sigma = f_{dd}$;
\item if $d'=d-2$ then, for $1\leq i \leq d'$, $f_{i, d-1}^{\sigma} = f_{\res_{d'}(i+1),d},f_{i, d}^{\sigma} = f_{\res_{d'}(i+1),d-1}$, and $f_{d-1,d-1}^{\sigma} = f_{dd}, f_{d-1,d}^{\sigma} = f_{d-1,d}$, and $f_{dd}^{\sigma} = f_{d-1,d-1}$.
\end{enumerate}
If $\mathscr{F}$ has a partial labelling $\{f_{ij} \mid (i,j)\in\App_W\} \subset \mathscr{F}$, then we say that $\mathscr{F}$ satisfies the $\sigma$-relations for $\App_W$ if the relations hold for all $(i,j)\in \App_W$.
\end{definition}
\begin{lemma}\label{eij relations SymSquare}
Let $G$ be a classical group as in one of the lines of Table \ref{SpecialTable}, let $s\in G$ be a special element as defined in Definition \ref{SpecialDefn}, and let $\mathscr{E}(s,S^2(V)) = \{ e_{ij}:=e_{s,S^2(V), ij} \mid 1\leq i \leq j\leq d\}$ as defined in Lemma \ref{EigenstructureOnWKSymSquare}. Then $\mathscr{E}(s,S^2(V))$ satisfies the $\sigma$-relations for $S^2(V)$.
\end{lemma}
\begin{proof}
The relations follow immediately from the fact that, by Lemma \ref{SpecialEigLem}, the $\sigma$-relations for $V$ (as in Definition \ref{Sigma Relations for V}) hold for $\{e_{i} \mid 1\leq i\leq d\}$.
\end{proof}
\begin{lemma}\label{Fi SymSquare 1}
Let $G$ be a classical group as in one of the lines of Table \ref{SpecialTable}, let $s\in G$ be a special element as defined in Definition \ref{SpecialDefn}, let $W= S^2(V)$, and let $\mathscr{E}=\mathscr{E}(s,S^2(V))=\{e_{ij}\mid 1\leq i\leq j\leq d\}$ be as in Definition \ref{Sigma Relations SymSquare}.\\\\
Suppose that there exists a basis $\mathscr{F}_{S^2(V)}:=\mathscr{F}(s,S^2(V),\mathscr{E}):= \{f_{ij}:=f_{s,V,\mathscr{E},ij}\mid 1\leq i\leq j\leq d \}$ for $W_K$ such that, for every pair $(i,j)$ with $1\leq i\leq j\leq d$, we have that $f_{s,V,\mathscr{E},ij}\in \<{e_{ij}}$, and $\mathscr{F}_{S^2(V)}$ satisfies the $\sigma$-relations for $S^2(V)$ as defined in Lemma \ref{EigenstructureOnWKSymSquare}.\\\\
Then there exists an extension field $L$ of $K$ of degree at most $2$, a basis $\mathscr{F}_V= \mathscr{F}(s,V,\mathscr{F}_{S^2(V)}):=\{f_{i}\mid 1\leq i\leq d\}$ for $V_L$, and a set $\mathscr{C}=\mathscr{C}(s,\mathscr{F}_{S^2(V)},\mathscr{F}_V):=\{c_{ij} \mid 1\leq i\leq j\leq d\}\subseteq L$, such that $\mathscr{F}_V$ satisfies the almost-$\sigma$-relations for $V$, and for $1\leq i\leq j\leq d$, we have 
\[
f_{ij} = c_{ij} \left(f_{i}\otimes f_j + (1-\delta_{ij})f_j\otimes f_i\right).
\]
Moreover, we have that $c_{11}=1$; and if $d'< d$, we have $c_{1j}=1$ for $j>d'$.
\end{lemma}
\begin{proof}
Let $\mathscr{E}_V:=\{e_i\}$ be as defined in Lemma \ref{SpecialEigLem}. Then since each $f_{ij}$ is a scalar multiple of $e_{ij}$, there exist constants $c_{ij}'\in K$ such that $f_{ij} = c_{ij}'e_{ij}$. 
Suppose that for every $i$, we set $f_i = a_i e_i$, for $a_i\in L$. Then for every $i,j$, we have that $f_{i}\otimes f_j = (a_i e_{i})\otimes (a_j e_j) = a_ia_j (e_i \otimes e_j) $, and so 
\[
f_{i}\otimes f_j + (1-\delta_{ij})f_j\otimes f_i = a_i a_j\left(e_{i}\otimes e_j + (1-\delta_{ij})e_j \otimes e_i \right) = a_i a_j e_{ij}.
\] 
For $1\leq i\leq j\leq d'$, set $c_{ij} := c_{ij}' (a_ia_j)^{-1} \in L$: then
\[
c_{ij} \left(f_{i}\otimes f_j + (1-\delta_{ij})f_j\otimes f_i\right) = c_{ij}'  e_{ij} = f_{ij}.
\] 
This holds for all choices of $\{a_i\}$, and so we have a great deal of freedom. By \cite[Theorem 2.14]{LidlNiederreiter}, we may choose a square root $x$ of $c_{11}'$ in the extension field $L/K$. \\\\
For $1\leq i\leq d'$, set $a_i = x^{q^{i-1}}$: then $a_1^2=x^2=c_{11}'$, and so $c_{11} = c_{11}' a_1^{-2} = 1$, and when $i < d'$, we have $f_i^{\sigma} = (a_i e_i)^{\sigma} = (x^{q^{i-1}})^q e_{i+1}  = a_{i+1} e_{i+1} = f_{i+1}$. When $i=d$, we have that $e_d^{\sigma}=e_1$, and so $f_i^{\sigma} = a_d^q f_1 \in \<{f_1}$, and so the almost-$\sigma$-relations hold for $i<d'$. The other relations follow in a similar way.\\\\
If $d' < d$, then for $j > d'$, set $a_j = c_{1j}' a_1^{-1}$: then $a_1 a_j = c_{1j}'$, and we have $c_{1j} = c_{1j} (a_1a_j)^{-1} = 1$ as required.
\end{proof}
The purpose of Lemma \ref{Fi SymSquare 1} is to allow a safe `transition' between an action of $g\in G$ on $S^2(V)_K$ to an action on $V_L$ (although we will later show that this action remains within the confines of $V_K$). However, it depends on our ability to find the constants $c_{ij}$, and this is not necessarily possible. There is, at one point in the procedure, a square root to be taken, and so a choice must be made, and a sign ambiguity introduced. The following result permits us to choose either path without regret.
\begin{lemma}\label{Sign Ambiguity Lemma SymSquare}
Let $G$ be a classical group as in one of the lines of Table \ref{SpecialTable}, let $s\in G$ be a special element as defined in Definition \ref{SpecialDefn}, let $W= S^2(V)$, and let $\mathscr{E}_{S^2(V)}=\{e_{ij}\mid 1\leq i\leq j\leq d\}$ be as defined in Lemma \ref{EigenstructureOnWKSymSquare}.\\\\
Suppose that there exists $\mathscr{F}_{S^2(V)} := \{f_{s,V,\mathscr{E},ij}\mid 1\leq i\leq j\leq d \}, L, \mathscr{F}_V:=\{f_i \mid 1\leq i\leq d\}, \mathscr{C}(s,\mathscr{F}_{S^2(V)},\mathscr{F}_V)=\{c_{ij}\mid 1\leq i\leq  j\leq d\}$ as defined in Lemma \ref{Fi SymSquare 1}. \\\\
Then define a basis $\mathscr{F}_V^-:= \{f_i^- \mid 1\leq i \leq d\}$ and a set of constants $\mathscr{C}^- = \mathscr{C}(s,\mathscr{F}_{S^2(V)},\mathscr{F}_V)^-:=\{c_{ij}^- \mid 1\leq i\leq j\leq d\}\subset L$, as follows:
\begin{enumerate}
\item for $1\leq i\leq d'$, let $f_i^- = (-1)^{i+1} f_i$; and for $i > d'$ let $f_i^- = f_i$; and
\item for $1\leq i\leq j\leq d'$, let $c_{ij}^- := (-1)^{j-i} c_{ij}$; for $1\leq i \leq d'$, $j>d'$, let $c_{ij}^- := (-1)^{i+1} c_{ij}$; and for all other $(i,j)$ let $c_{ij}^- := c_{ij}$.
\end{enumerate}
Then for $1\leq i\leq j \leq d$, we have
\[
f_{ij} = c_{ij}^- \left(f_{i}^-\otimes f_j^- + (1-\delta_{ij})f_j^-\otimes f_i^-\right).
\]
\end{lemma}
\begin{proof}
Observe that, for $1\leq i\leq j\leq d'$, we have 
\[
c_{ij}^- \left(f_{i}^-\otimes f_j^- + (1-\delta_{ij})f_j^-\otimes f_i^-\right) 
= (-1)^{2j+2}c_{ij} \left(f_{i}\otimes f_j + (1-\delta_{ij})f_j\otimes f_i\right), 
\]
and since $2j+2$ is even, this is precisely $f_{ij}$.
When $i \leq d', j > d'$, we have 
\[
c_{ij}^-  \left(f_{i}^-\otimes f_j^- + (1-\delta_{ij})f_j^-\otimes f_i^-\right) = (-1)^{2(i+1)}c_{ij} \left(f_{i}\otimes f_j + (1-\delta_{ij})f_j\otimes f_i\right),
\]
which again is equal to $f_{ij}$ since $2(i+1)$ is even. If $d'<i\leq j\leq d$ then the assertion holds trivially by the definitions.
\end{proof}
Lemma \ref{Sign Ambiguity Lemma SymSquare} allows us to `err' in our search for the values of the $c_{ij}$, so long as we `accidentally' find $c_{ij}^-$: in that case, \texttt{FindPreimage} will return the action of $g\in G$ with respect to $\mathscr{F}_V^-$ instead of $\mathscr{F}_V$, a mistake which is irrelevant to us, and by Lemma \ref{Fi SymSquare 1} above, is unavoidable. Note that we are only permitted to make \emph{one} mistake: we must compute all of either $\mathscr{C}$ or $\mathscr{C}^-$, and we cannot `mix and match'.
\begin{lemma}\label{Fi SymSquare 2}
Let $G$ be a classical group as in one of the lines of Table \ref{SpecialTable}, let $s\in G$ be a special element as defined in Definition \ref{SpecialDefn}, let $W = S^2(V)$. Let $\mathscr{E}_{S^2(V)}=\{e_{ij}\mid 1\leq i\leq j\leq d\}$ be as defined in Lemma \ref{EigenstructureOnWKSymSquare}, and let $\mathscr{F}_V$, $\mathscr{C}$ be defined as in Lemma \ref{Fi SymSquare 1}, and let $\mathscr{F}_V^-, \mathscr{C}^-$ be as defined in Lemma \ref{Sign Ambiguity Lemma SymSquare}.\\\\
Then for $1\leq i\leq j\leq d'-1$, we have $c_{ij}^q = c_{i+1,j+1}$ and $(c_{ij}^-)^q = c_{i+1,j+1}^-$, and hence for $1\leq i\leq d'$, we have $c_{ii}=c_{ii}^-=1$. \\\\
Moreover, for every $g\in G$, and for all pairs $(i,j),(k,\ell) \in\App_W$, we have the following, where $\kappa_{ij,k\ell} :=g_{f_{ij}f_{k\ell}}, a_{ij} = g_{f_if_j}, a_{ij}^- = g_{f_i^- f_j^-}$:
\begin{enumerate}
\item The \emph{Basic Equations in the Symmetric Square Case} hold for $(\kappa_{ij,k\ell}), \mathscr{C},(a_{ij}), \mathscr{C}^{-}, (a_{ij}^-)$: 
\begin{equation}\label{BasicEquationSymSquare}
\kappa_{ij,k\ell} = \frac{c_{ij}}{c_{k\ell}}(a_{ik}a_{j\ell} + (1-\delta_{ij})a_{i\ell}a_{jk})= \frac{c_{ij}^-}{c_{k\ell}^-}(a_{ik}^-a_{j\ell}^- + (1-\delta_{ij})a_{i\ell}^-a_{jk}^-)
\end{equation}
\item for $1\leq i,j\leq d'-1$, we have $a_{ij}^q = a_{i+1,j+1},(a_{ij}^-)^q = a_{i+1,j+1}^- $.
\end{enumerate}
\end{lemma}
\begin{proof}
Since $f_{i}^{\sigma} = f_{i+1}$ for all $1\leq i\leq d'-1$, and since $\delta_{ij}=\delta_{i+1,j+1}$, we have, for $1\leq i\leq j \leq d'-1$, that
\[
\begin{array}{rl}
 f_{ij}^{\sigma} 
 &= \left(c_{ij} \left(f_{i}\otimes f_j + (1-\delta_{ij})f_j\otimes f_i\right)\right)^{\sigma} \\
 &= c_{ij}^q \left(f_{i+1}\otimes f_{j+1} + (1-\delta_{i+1,j+1})f_{j+1}\otimes f_{i+1}\right)\\ 
 &= c_{ij}^q \left(c_{i+1,j+1}^{-1} f_{i+1,j+1}\right),
\end{array}
\]
and so $c_{ij}^q=c_{i+1,j+1}$ since, by definition, we have that $f_{ij}^{\sigma} = f_{i+1,j+1}$. Since, by Lemma \ref{Fi SymSquare 1}, $c_{11}=1$, we have that $c_{ii}=1$ for all $i$. The relations on $c_{ij}^-$ follow immediately since $c_{ij}^- = (-1)^{j-i}c_{ij}$.\\\\
(i) follows immediately from Lemmas \ref{Basic Equations Lemma}. For (ii), since $f_i^{\sigma} = f_{i+1}$ for $1\leq i\leq d'-1$, and since $f_i^{\sigma} \not\in \<{f_2, \ldots, f_{d'}}$ for $i \geq d'$, we have that 
\[
f_i^{g\sigma} = \left(\sum_{j=1}^d a_{ij} f_j \right)^{\sigma} = \sum_{j=1}^{d'-1} a_{ij}^q f_{j+1} +  v,
\]
for some $v\not\in\<{f_2, \ldots, f_{d'}}$.
On the other hand
\[
f_i^{\sigma g} = f_{i+1}^g = \sum_{j=1}^d a_{i+1,j} f_j=\sum_{j=0}^{d-1} a_{i+1,j+1} f_{j+1} ,
\]
and so, since $f_i^{g\sigma} = f_i^{\sigma g}$, by equating coefficients of $f_{j+1}$ for $1\leq j\leq d'-1$, we have $a_{ij}^q = a_{i+1, j+1}$ as required.
A similar argument shows that the relations hold among the $a_{ij}^-$, since $(f_i^-)^{\sigma} = ((-1)^{i+1})^{\sigma}f_i^{\sigma} =  (-1)^{i+1}f_{i+1} = -f_{i+1}$.
\end{proof}
\begin{lemma}\label{Final Lem SymSquare}
Let $G$ be a Classical Group, let $V$ be the natural $FG$-module, and let $W$ be an $FG$-module such that $W = (S^2(V)^*)^{\varphi}$, where $S^2(V)^*$ is an irreducible section of $S^2(V)$ of codimension at most $2$, and $\varphi$ is an isomorphism of $FG$-modules, and let $\nu$ be a homomorphism of $FG$-modules such that $S^2(V)^*\leqslant S^2(V)^{\nu}$.\\\\
Let $s\in G$ be a special element, as in Definition \ref{SpecialDefn}, let $\{\ell_{ij} \mid 1\leq i\leq j\leq d'\}$ be the eigenvalues of $s$ in $S^2(V)_K$, as in Lemma \ref{EigenstructureOnWKSymSquare}. Suppose that $\mathscr{F}_W = \{f_{W,ij} \mid (i,j)\in\App_W \} \cup \mathscr{F}'$ is a basis of $W_K$ such that, for every $(i,j)\in \App_W$, we have that $f_{W,ij}$ is an $\ell_{ij}$-eigenvector for $s$ in $W_K$, and $\mathscr{F}_W$ satisfies the $\sigma$-relations for $S^2(V)$ for all $(i,j)\in\App_W$ as in Definition \ref{eij relations SymSquare}.\\\\
Then there exists a basis $\mathscr{F}_{S^2(V)}= \{\hat{f}_{ij}\mid 1\leq i\leq j\leq d \}$ of $S^2(V)_K$ such that $\mathscr{F}_{S^2(V)}$ satisfies the $\sigma$-relations for $S^2(V)$, and for every $(i,j)\in\App_W$, the following hold:
\begin{enumerate}
\item $\hat{f}_{ij}^{\nu} \in S^2(V)^*$;  
\item $\hat{f}_{ij}^{\nu \varphi} = f_{ij}$; and
\item for every $g\in G$ and for all $(i,j),(k,\ell)\in\App_W$, we have that $g_{\hat{f}_{ij}\hat{f}_{k\ell}} = g_{f_{ij} f_{k\ell}}$.\end{enumerate}
\end{lemma}
\begin{proof}
Note first that $\nu$ is either the identity map, or the projection of $S^2(V)$ onto a quotient by a subspace fixed pointwise by $G$. For each $(1,j)\in \App_W$, choose a preimage of $f_{1j}^{\varphi^{-1}}$ under $\nu$, and set $\hat{f}_{1j}$ to be this preimage. Then for $(i,j)\in\App_W$ with $2\leq i\leq j\leq d'$, set $\hat{f}_{ij} := \hat{f}_{i-1,j-1}^{\sigma^{i-1}}$.
If $d' = d-1$, then for $2\leq i \leq d'$, set $\hat{f}_{id}:= \hat{f}_{i-1,d}^{\sigma}$.
If $d' = d-2$, then for $2\leq i \leq d'$, set $\hat{f}_{id}:= \hat{f}_{i-1,d-1}^{\sigma}$, and set $\hat{f}_{i,d-1}:= \hat{f}_{i-1,d}^{\sigma}$. Choose a preimage of $f_{d-1,d-1}^{\varphi^{-1}}$ under $\nu$, set $\hat{f}_{d-1,d-1}$ to be this preimage, and set $\hat{f}_{dd} = \hat{f}_{d-1,d-1}^{\sigma}$.
Then $\hat{f}_{ij}$ has been defined for all pairs $(i,j)\in\App_W$: for the remaining pairs with $1\leq i\leq j\leq d$, choose $\hat{f}_{ij}$ such that they satisfy the $\sigma$-relations for $S^2(V)$.\\\\
Now for all $1\leq i\leq j\leq d$, we have that $\hat{f}_{ij}$ is an $\ell_{ij}$-eigenvector for $s$ in its action on $S^2(V)$, since the maps $\nu, \varphi$ preserve eigenstructure, and since the action of $\sigma$ maps $\ell_{ij}$-eigenvectors to $\ell_{ij}^q$-eigenvectors. By Lemma \ref{EigenstructureOnWKSymSquare}, for every $(i,j)\in\App_W$, either $G=\SU(d,q)$, or $\ell_{ij}\neq 1$. In the former case, $\nu$ is the identity map, and so $\hat{f}_{ij}^{\nu} \in S^2(V)^*$. In the latter case, $\hat{f}_{ij}$ is an $\ell_{ij}$-eigenvector for $\ell_{ij}\neq 1$, and so is not fixed by the action of $s$, and so since $S^2(V)^*$ is the kernel a linear form $T$, we have, by Lemma \ref{MissingPieces}, that $\hat{f}_{ij}^{\nu}\in S^2(V)^*$.\\\\
Since $\sigma$ commutes with $\varphi,\nu$, we have, for $(i,j)\in\App_W$ with $2\leq i\leq j\leq d'$, that $\hat{f}_{ij}^{\nu \varphi} = (\hat{f}_{i-1,j-1}^{\sigma})^{\nu \varphi} = (\hat{f}_{i-1,j-1}^{\nu \varphi})^{\sigma} = f_{i-1,j-1}^{\sigma} = f_{ij}$. By the same argument we have that $\hat{f}_{ij}^{\nu\varphi} = f_{ij}$ for the remaining $(i,j)\in\App_W$, and so (ii) holds. (iii) then follows.
\end{proof}
Lemma \ref{Final Lem SymSquare} `lifts' us from a basis of $W_K$ to a basis of $S^2(V)_K$: combining this with Lemma \ref{Fi SymSquare 2} `decomposes' into one of two bases for $V_L$ for which the Basic Equation holds whenever $(i,j),(k,\ell)\in \App_W$. This set of equations is the tool for constructive recognition.
\begin{corollary}\label{Basic Equations Corollary SymSquare}
Let $G$ be a Classical Group, let $V$ be the natural $FG$-module, and let $W$ be an $FG$-module such that $W = (S^2(V)^*)^{\varphi}$, where $S^2(V)^*$ is an irreducible section of $S^2(V)$ of codimension at most 2, and $\varphi$ is an isomorphism of $FG$-modules.\\\\
Let $s\in G$ be a special element, as in Definition \ref{SpecialDefn}, let $\{\ell_{ij} \mid 1\leq i\leq j\leq d'\}$ be the eigenvalues of $s$ in $S^2(V)_K$, as in Lemma \ref{EigenstructureOnWKSymSquare}. Suppose that $\mathscr{F}_W = \{f_{W,ij} \mid (i,j)\in\App_W \} \cup \mathscr{F}'$ is a basis of $W_K$ satisfying the conditions in Lemma \ref{Final Lem SymSquare}.\\\\
Then there exists a field extension $L$ of $K$ of degree at most $2$, a basis $\mathscr{F}_V= \mathscr{F}(s,V,\mathscr{F_W}) = \{f_{V,i} \mid 1\leq i\leq d\}$ of $V_L$, a set of constants $\mathscr{C}=\{c_{ij}\mid (i,j)\in\App_W\}$, a basis $\mathscr{F}_V^-= \{f_{V,i}^- \mid 1\leq i\leq d\}$ and constants $\mathscr{C}^-=\{c_{ij}^-\mid (i,j)\in\App_W\}$ as defined in Lemma \ref{Sign Ambiguity Lemma SymSquare}, such that, for every $(i,j),(k,\ell)\in\App_W$ and for every $g\in G$, the following hold, where $\kappa_{ij,k\ell} = g_{f_{W,ij}f_{W,k\ell}}, a_{ij} = g_{f_{V,i}f_{V,j}}, a_{ij}^- = g_{f_{V,i}^- f_{V,j}^-}$:
\begin{enumerate}
\item The Basic Equation (\ref{BasicEquationSymSquare}) holds; and
\item for $1\leq i,j\leq d'-1$, we have $a_{ij}^q = a_{i+1,j+1}$.
\end{enumerate}
Moreover, we have $c_{ii}=c_{ii}^-=1$ for $1\leq i\leq d'$, and if $d'<d$, we have $c_{1j}=1$ for $d'<j \leq d$.
\end{corollary}
\begin{proof}
By Lemma \ref{Final Lem SymSquare}, the basis $\mathscr{F}_W$ gives rise to a basis $\mathscr{F}$ of $S^2(V)$ satisfying the conditions of Lemma \ref{Fi SymSquare 2}: combining these two results, the result follows.
\end{proof}
\section{The Algorithm}
In this Section we detail the steps in \texttt{Initialise} and \texttt{FindPreimage}. We first must find a Special Element $s$ (by random search), and then find the eigenvalues of $s$ in its action on $W_K$, a basis $\mathscr{F}_W:=\{f_{ij} \mid (i,j)\in\App_W \} \cup \mathscr{F}'$ for $W_K$ of $s$-eigenvectors, and constants $c_{ij}$ for certain values of $i,j$ satisfying the conditions of Lemma \ref{Final Lem SymSquare}. Since $S^2(V)$ contains the Alternating Square $\wedge^2(V)$ when $q$ is even (and so is irreducible in a nontrivial way) \emph{we assume that $q$ is odd}.\\\\
Parts of the procedure work for all odd $q$, but ultimately \texttt{Initialise} can be completed only when $q \geq 5$ in certain cases (due to the exceptions in Lemma \ref{EigenstructureOnWKSymSquare}). The deciding factor is whether or not $(1,1)\in\App_W$: when $G = \{\Sp(d,3), \SO^{\epsilon}(d,3)\}$ things break down.
\subsection{Finding the Special Element}\label{Find Special Element Section}
In the \texttt{FindSpecialElement} procedure, we assume that we have access to an oracle providing random elements of $H=\<{X}$: we denote by $\xi_{H}$ the time required to produce such elements. In practice, we use the built-in functions of GAP and MAGMA to produce random or pseudorandom elements (in the GAP code, we use the built-in function \texttt{PseudoRandom}, and in the MAGMA implementation, the builtin function \texttt{Random}).
We require a polynomial-time `test for suitability', which takes a matrix $g\in G$ as input and returns \texttt{TRUE} if $g$ is a special element, and \texttt{FALSE} otherwise. Of course, computing directly the eigenvalues of an element and checking that their number is sufficiently high would work (and would ultimately not effect the analysis of the procedure's complexity), but we wish to discard unsuitable choices as quickly as possible.
The name \texttt{FindSpecialElement} is slightly inaccurate: we require our elements to be $\ppd(d,q;d')$-elements with sufficiently many $1$-dimensional eigenspaces (specifically, we ask that the $\ell_{ij}$-eigenspace be $1$-dimensional for all $(i,j)\in\App_W$). Such elements form a superset of the special elements: in Section \ref{Section:Quokka}, we find lower bounds on the proportion of special elements in $G$, which automatically gives a lower bound on the probability that a randomly chosen element of $H$ will have the desired properties.\\\\
In order to find a special element in the case $G=\SU(d,q)$ with $d$ even, we do not simply search for them: instead, we search for a more abundant type of element from which a special element can be constructed. Note that here and henceforth, we consider $\SU(d,q)$ to be a subgroup of $\SL(d,q)$, defined only when $q$ is a square.
\begin{definition}\label{PreSpecial Defn}
Let $G=\SU(d,q)$, with $d$ even, and let $d'=d-1$. Then $s\in G$ is called a \emph{pre-special element} of $G$ if $o(s) = \sqrt{q}^{d'}+1$.
\end{definition}
\begin{lemma}\label{special is a power of a prespecial}
Let $G=\SU(d,q)$, with $d$ even and $d\geq 4$, and let $s$ be a pre-special element of $G$. Then $s^{\sqrt{q}+1}$ is a special element.
\end{lemma}
\begin{proof}
This follows immediately from the definitions of special and pre-special elements: since $\sqrt{q}+1$ divides $o(s)$ the order of $s^{\sqrt{q}+1}$ is $\frac{\sqrt{q}^{d'}+1}{\sqrt{q}+1}$. Also, any primitive prime divisor $r$ of $\sqrt{q}^{d'}-1$ must be coprime to $\sqrt{q}+1$, since $(\sqrt{q}+1) \mid q-1$, and so $r\nmid (\sqrt{q}+1)$: thus $r\mid o(s^{\sqrt{q}+1})$, and so $s^{\sqrt{q}+1}$ is a $\ppd(d,q;d')$-element.
\end{proof}
As part of our test for specialness, we require a polynomial-time test for whether an element $s\in G$ has order divisible by a primitive prime divisor of $q^{d'}-1$. Since we will know the eigenvalues of $s$ when the time comes, it is cheaper to decide if the order of an eigenvalue is divisble by a primitive prime divisor $r$ of $q^{d'}-1$: since $o(\lambda)\mid o(s)$ for every eigenvalue $\lambda$ of $s$, if $r\mid o(\lambda)$ then $s$ is a $\ppd(d,q;d')$-element.
\begin{lemma}\label{TestPPDness Lemma}
If $(q,d') = (2,6)$, set $m:=21$. If $(q,d')=(p,2)$ for $p$ a Mersenne prime, then set $m := p-1$. For all other pairs $(q,d')$ with $q$ a prime power and $d'>2$, set 
\[
m := \prod_{\substack{j\mid d' \\ j < d'}} \frac{d'}{j}(q^j-1).
\]
Suppose that $\lambda\in K=\F{q^{d'}}$, and $\lambda^m \neq 1$. Then $\lambda\in G$ has order divisible by a primitive prime divisor of $q^{d'}-1$.
\end{lemma}
\begin{proof}
For an integer $x$, denote by $(x)_r$ the $r$-part of $x$, that is, the largest power of $r$ dividing $x$. Suppose that $r$ is a prime divisor of $o(s)$, and suppose that $e$ is the smallest integer such that $r \mid q^{e}-1$. Suppose that $1\leq e < d'$ (that is, $r$ is \emph{not} a primitive prime divisor of $q^{d'}-1$, and let $(q^{e}-1)_r = r^{t}$. Suppose that $r \neq 2$: then by \cite[Lemma 4.1(iii)]{niemeyer2012abundant}, we have that $(q^{d'}-1)_r = r^{t} \left(\frac{d}{e}\right)_r= (q^e-1)_r\left(\frac{d}{e}\right)_r$. Thus the $r$-part of $q^{d'}-1$ divides the $j=e$ term in the product defining $m$, and so $r$ does not divide the order of $\lambda^m$.\\\\
Suppose that $r=2$: then $q$ is odd, since if $q$ is even then $r\nmid q^{d'}-1$. If $d$ is even, then $q^{d'}-1 = (q^{d'/2}-1)(q^{d'/2}+1)$, and since $q$ is odd, we have $q^{d'/2} \equiv 1$ modulo $4$, and so $(q^{d'/2}+1)_2 = 2$. Then $(q^{d'}-1)_2 = 2(q^{d'/2}-1)_2$, and so the $j=d'/2$ term of the product defining $m$ is divisible by the $2$-part of $o(\lambda)$. It follows that $2$ does not divide the order of $\lambda^m$. If $d$ is odd, then $(q^{d'}-1)/(q-1) = 1+ \cdots + q^{d'-1}$ is the sum of an odd number of odd numbers, and so is odd. That is, $(q^{d'}-1)_2 = (q-1)_2$. Thus the $j=1$ term of the product defining $m$ is divisible by the $2$-part of $q^{d'}-1$, and so $2$ does not divide the order of $\lambda^m$.\\\\
Then since $o(\lambda^m)\mid o(\lambda)$, the only prime divisors of $o(\lambda^m)$ are the primitive prime divisors of $q^{d'}-1$ which divide $o(\lambda)$, and the result follows. In the exceptional cases (when $(q,d') = (2,6)$ or $(p,2)$ for $p$ a Mersenne prime) the result follows by a similar argument.
\end{proof}
\begin{remark}\label{TestPPDness remark}
Since the number of divisors of $d'$ is less than $2 \sqrt{d'}$, we have that $m = O(q^{d' + d'(2\sqrt{d'})}) = O(q^{d^{3/2}})$. Then by applying Lemma \ref{TestPPDness Lemma} to a known eigenvalue of $s$ in $K$, we can decide if $s$ is a $\ppd(d,q;d')$-element in $O(\rho_{q^{d'}} \log m) = O(\rho_{q^{d'}} \log (q^{d^{3/2}})) = O(\rho_{q^{d'}} d^{3/2} \log q)$ time.\\\\
Note that this task requires that we can completely factorise $d'$: we do not concern ourselves with the cost of non-field-operations. Since our computations take place in a field of size $q^{d'}$, the cost of factoring $d'$ will be small compared to the cost of field operations.
\end{remark}
The biggest speedup we can perform on the test for specialness is to avoid factoring the characteristic polynomial completely over $K$ if it is unnecessary: in particular it follows from the results below about the orbits of eigenvalues over $K$ that the characteristic polynomial of a special element has no irreducible factors over $F$ of any degree other than $1, 2, \frac{d'}{2}$ or $d'$, and in fact we know explicitly their distributions (which depend on the case). With this knowledge in our hand, we can eliminate most unsuitable candidates beforehand, and not waste our time computing the eigenvalues over $K$.  Recall that $\App_W(s)$ is the set of pairs $(i,j)$ such that the $\ell_{ij}$-eigenspace of $s_{W_K}$ is $1$-dimensional.
\begin{algorithm}
\caption{FindSpecialElement}\label{FindSpecialElement}
\begin{flushleft}
\textbf{Input:} A set $X \subseteq \GL(n,q)$, such that $H=\<{X}$ generates a nontrivial section $W$ of $S^2(V)$, represented as $n\times n$ matrices over $F = \F{q}$, and an acceptable probability of failure $\epsilon \in (0,1)$. \\ 
\textbf{Output:} A $\ppd(d,q;d')$-element of $H$, together with its (unlabelled) eigenvalues $\ell_{ij}$, separated into $\sigma$-orbits (where $\sigma$ is the Frobenius automorphism $x\mapsto x^q$).\\
\textbf{Procedure:}
\end{flushleft}
\begin{enumerate}
\item Set $T:= \ceiling{\frac{2}{P}\log(\epsilon^{-1})}$, where $P$ is the lower bound for the proportion $|S|/|G|$ given in Table \ref{TableOfWeylGroups} below.
\item If more than $T$ random elements of $H$ have been requested, then return \fail. Otherwise, choose a random element $g\in H$, and compute the characteristic polynomial $c_{g}(t)$ of $g$.
\item Compute the square-free factorisation of $c_g(t)$ (see \cite[Section 4.6]{knuth1997volume}). If $c_g(t)$ has a square divisor which is not a power of $(t-1)$, then discard $g$ and return to (ii).
\item Compute the distinct-degree factorisation of $c_g(t)$ (see \cite[Algorithm D]{kaltofen1998subquadratic}), which yields the number of irreducible factors of each degree $d',d'/2, 2,1$. If the degrees are not correct, then discard $g$ and return to (i); if they are, then $g$ has the correct number and arrangement of orbits of eigenvalues in $W_K$.
\item Compute the distinct linear factors of $c_g(t)$ over $K$ (using, for example, the algorithm of Beals et al. \cite[Lemma 4.6]{beals2005constructive}), and hence the eigenvalues of $g$ over $K$. For a zero $\beta\in K$ of one of the irreducible divisors of $c_g(t)$ of largest degree, compute $\beta^m$, for $m$ as in Lemma \ref{TestPPDness Lemma}. If the value is $1$, or if the computation of linear factors returns \texttt{FAIL}, then discard $g$ and return to (ii). 
\item In the case $G=\SU(d,q)$ for $d$ even, if $s$ does \emph{not} have $1$ as an eigenvalue, compute the $(q+1)$st power of each eigenvalue. If these powers are all distinct, then return $g^{q+1}$ and its eigenvalues. If not, then return to (ii). 
\item In all other cases, return $g$ and its eigenvalues over $K$. 
\end{enumerate}
\end{algorithm}
\begin{proposition}\label{FindSpecialElement Prop}
Algorithm \ref{FindSpecialElement} is a Las Vegas algorithm which returns, with probability at least $1-\epsilon$, an element $g$ of $H$ such that the following hold:
\begin{enumerate}
\item $g$ is a $\ppd(d,q;d')$-element;
\item There exists a labelling of the eigenvalues of $g$ as $\{\ell_{ij}\mid (i,j)\in\App_W\}$, such that the eigenspace of $\ell_{ij}$ is $1$-dimensional for all $(i,j)\in\App_W$ (see Lemma \ref{AppLem SymSquare});
\end{enumerate}
and has complexity
\[
O\left( (\xi_{H} + \rho_q d^3(d^3 + \log q) + \rho_{q^{d'}} d^3 \log^2 d \log(dq)) \frac{1}{P} \log \epsilon^{-1}\right),	
\]
where $P$ is the proportion of special elements in $G$.
In particular, using the bound $P > \frac{1}{9d^2\log^2 q}$ (see Section \ref{Section:Quokka} and Table \ref{TableOfWeylGroups}), we have that $\frac{2}{P} < \frac{9}{2} d^2\log^2 q) $, and so Algorithm \ref{FindSpecialElement} has complexity
\[
O\left( (\xi_{H} + \rho_q d^3(d^3 + \log q) + \rho_{q^{d'}} d^3 \log^2 d \log(dq)) d^2 \log^2 q{\log\epsilon^{-1}}\right)	
\]
\end{proposition}
\begin{proof}
It is possible that we may fail to detect the unsuitability of an element until the very last test in (vi), and so in the worst case, we must run (i)-(vi) on $T$ matrices. Step (ii) costs $O(\xi_{H} + \rho_q d^6)$. Step (iii) costs $O(\rho_q d^3 \log q)$ and (iv) runs faster than (iii).\\\\
We can find the distinct linear factors of the characteristic polynomial using the Las Vegas algorithm of \cite{beals2005constructive} in time 
\[
O\left(\rho_{q^{d'}} n \log n \log(nq^d) \log\log n \log\epsilon^{-1}\right) = O\left(\rho_{q^{d'}} d^3 \log^2 d \log(dq)\log\epsilon^{-1}\right).
\]
and testing whether $\beta^m=1$ requires time $O(\rho_{q^{d'}} d^{3/2} \log q)$ by Remark \ref{TestPPDness remark}. Thus each test has total worst-case cost
\[
O\left(\rho_q (d^6 + d^3\log q) + \rho_{q^{d'}} d^3 \log^2 d \log(dq)\log\epsilon^{-1}\right).
\]
The result then follows since every special element will pass, and so $T$ tests is sufficient to ensure that the probability of failure is at most $\epsilon$ (note that the $\epsilon$ introduced by the factoring algorithm is a different value, and we may have to split our error probability between the two: this is a mere technicality and we omit dealing with it for the sake of space and time).
\end{proof}
\subsection{Labelling the Eigenvalues $\ell_{ij}$}\label{LabelEigenvalues SymSquare Section}
The goal of this section is to present the family of \texttt{LabelEigenvalues} procedures which, given the eigenvalues and corresponding eigenspaces of a special element $s\in H$ in $W_K$, produce a valid labelling of their orbits according to the structure given in Lemma \ref{EigenstructureOnWKSymSquare}.
\begin{definition}\label{valid labelling}
An assignment $(i,j)\mapsto \ell_{ij}$ is called a \emph{valid labelling of eigenvalues} if there exists a set $\{\ell_{i} \mid 1\leq i\leq d\}$ such that, for every pair $1\leq i\leq j\leq d$, we have $\ell_{ij}=\ell_i\ell_j$. 
\end{definition}
A valid labelling of the $\ell_{ij}$ allows us to find eigenvectors $f_{ij}$ for $W_K$ satisfying the conditions of Lemma \ref{Final Lem SymSquare}. While naive searching would suffice to perform this task `quickly enough' (in the sense that this part of \texttt{Initialise} is not a bottleneck), nevertheless we employ shortcuts to speed up the process.\\\\
We proceed case by case, according to the value of $d'$ as defined in Table \ref{SpecialTable}: recall that $d' \in \{d-2,d-1,d\}$.
%
%
%
\subsubsection{The Case $d'=d$}
In this case, the eigenvalues of $s$ in its action on $W_K$ are, by Lemma \ref{EigenstructureOnWKSymSquare},
\[
\{ \ell_{ij} \mid 1\leq i \leq j \leq d\}.
\]
It follows directly from the definition that for every $i,j$, we have $\ell_{ii}\ell_{jj} = \ell_{ij}^2$ -- that is, the $\sigma$-orbit $\Omega$ containing $\ell_{11}$ has the property that the product of any two distinct members of $\Omega$ is the square of an eigenvalue in another orbit. Our procedure uses this property to find a suitable $\ell_{11}$ by eliminating those orbits which do not possess the property.\\\\
We begin by storing in memory the set of squares of the eigenvalues. Then choosing at random a candidate for $\ell_{11}$, we test whether $\ell_{11}^{1+q^{j-1}}$ (which is equal to $\ell_{11}\ell_{jj}$) lies in this set of squares for $2\leq j\leq d$. If this test fails for any $j$, we select another orbit and try again. Since the square root operation is very costly for large fields $K$, and since we do not need to know the square root of $\ell_{11}^{1+q^{j-1}}$ explicitly -- only whether or not it is one of the $\ell_{ij}$ -- the memory we use to hold this relatively small lookup table is a small price to pay for a much faster procedure.
{
\begin{algorithm}[ht]
\caption{LabelEigenvaluesSymSquare(d'=d)}\label{LabelEigenvaluesSymSquare(d'=d)}
\begin{flushleft}\textbf{Input:} A special element $s\in H$, and the eigenvalues of $s$ in its action on $W_K$, in the case that $d'=d$.\\
\textbf{Output:} A valid labelled set $\{\ell_{ij}\mid 1\leq i \leq j \leq d\}$ of eigenvalues, and a basis $\mathscr{F}_W=\{f_{ij}\mid (i,j)\in \App_W\}\cup \mathscr{F}'$ of $W_K$, satisfying the conditions in Lemma \ref{Final Lem SymSquare}.\\
\textbf{Procedure:}
\end{flushleft}
\begin{enumerate}
\item Compute the $q$th power of each eigenvalue, and sort them into ordered $\sigma$-orbits.
\item Compute the square of each eigenvalue, and store these in a list $\Sigma^2$, with a record of the correspondence between eigenvalues and their squares.
\item For each orbit $\Omega$ of eigenvalues, choose an element $\alpha \in\Omega$. For $2\leq k\leq d-1$, compute $\alpha^{q^{k-1}+1}$. If the result lies in $\Sigma^2$, find its square root and label it $\ell_{1k}$; if not, then discard $\Omega$ and choose another orbit. Once all of $\ell_{1k}$ have been labelled, proceed to (iv).
\item For $2\leq i < j\leq d$, label $\ell_{ij} = \ell_{i-1,j-1}^q$.
\item For each $(1,j)\in\App_W$, set $f_{1j}$ to be the eigenvector of $\ell_{1j}$ having a $1$ in its first nonzero entry.
\item For $(i,j)\in\App_W$ with $i\geq 2$, set $f_{ij} = f_{i-1,j-1}^{\sigma}$.
\item If necessary, extend $\{f_{ij}\mid (i,j)\in\App_W\}$ to a basis for $W_K$ in any way: these eigenvectors are of no consequence to us.
\end{enumerate}
\end{algorithm}
}
\begin{remark}\label{Labelling Eigs Remark}
\begin{enumerate}
\item In step (i) of \texttt{LabelEigenvaluesSymSquare(d'=d)} (and in subsequent \texttt{LabelEigenvalues} procedures outlined below) we do not simply compute the $\sigma$-orbits of eigenvalues, but retain a record of the $q$th power of each eigenvalue. In practice this is achieved by storing each orbit as an ordered list, with each entry the $q$th power of its predecessor. This step requires $O(d^2)$ $q$th-power computations, each with a cost of $O(\rho_{q^{d'}}\log q)$, and so the setup of this data structure has complexity $O(\rho_{q^{d'}} d^2\log q)$. Once this data structure has been set up, computation of $q^k$th powers of eigenvalues has zero cost (to find the $q^k$th power of an eigenvalue we simply move $k$ spaces down the list), and hence computation of $(q^k+1)$st powers can hence be performed with a single field operation.
\item In step (iii) of \texttt{LabelEigenvaluesSymSquare(d'=d)}, the computation of square roots is free, since in step (ii) we store a correspondence between eigenvalues and their squares. This has a relatively small memory cost, and saves a considerable amount of time, since taking a square root in $K$ has a cost of $O(\rho_{q^{d'}} d \log q)$.
\item In practice we perform the final step, of extending the partial basis $\{ f_{ij} \mid (i,j) \in\App_W \}$ to a basis for $W_K$, by computing a basis for the $1$-eigenspace of $s$ in $W_K$ (or, in fact, in $W$, since the $1$-eigenspace has a basis consisting only of $F$-vectors: the distinction is of little consequence).
\end{enumerate}
\end{remark}
\begin{proposition}\label{LabelEigenvaluesSymSquare(d'=d) Prop}
Algorithm \ref{LabelEigenvaluesSymSquare(d'=d)} (\texttt{LabelEigenvaluesSymSquare(d'=d)}) returns a valid labelling $\{\ell_{ij} \mid 1\leq i\leq j \leq d\}$ of the eigenvalues of $s$ on $W_K$ together with a basis $\mathscr{F}_W=\{ f_{W,ij} \mid (i,j)\in\App_W\} \cup \mathscr{F}'$ for $W_K$ satisfying the conditions in Lemma \ref{Final Lem SymSquare}; and has complexity 
\[
O(\rho_{q^{d'}} d^4 (d^3+ \log q)) ,
\]
where $\rho_{q^{d'}}$ is the cost of a field operation in $K=\F{q^{d'}}$.
\end{proposition}
\begin{proof}
Steps (iii)-(iv) yield a valid choice of $\ell_{11}$: setting $\ell_{1}$ to be a square root of this value and setting $\ell_i = \ell_1^{q^{i-1}}$ we have that $\ell_{ij}=\ell_i\ell_j$ -- that is, we have a valid labelling of the eigenvalues. Note that the orbit of the true value of $\ell_{11}$ must be tested (since all orbits are tried), and the choice within that orbit is unimportant (for choosing another element of the orbit simply relabels the $\ell_i$ by a cyclic permutation), and so the algorithm terminates after testing every orbit in the worst case. Since for $(i,j)\in\App_W$, $f_{ij}$ is an $\ell_{ij}$-eigenvalue, and $\mathscr{F}_W$ satisfies the $\sigma$-relations in the Symmetric Square case for $(i,j)\in\App_W$ (see Definition \ref{Sigma Relations SymSquare}) by the construction in step (vi), $\mathscr{F}_W$ satisfies the conditions of Lemma \ref{Final Lem SymSquare}.\\\\ 
Step (i) costs $O(\rho_{q^{d'}} d^2 \log q)$, by Remark \ref{Labelling Eigs Remark}(i). Steps (ii)-(iii) cost $O(d^2)$, since there are $O(d^2)$ squares to take in step (ii), and in the worst case there are $d$ orbits to try, and $d-1$ powers $\alpha^{q^k+1}$ to test. \\\\
Since the $q$th power of every eigenvalue is known (from (i)), step (iv) costs nothing: by labelling the first element in an orbit, we implicitly label the entire orbit (see Remark \ref{Labelling Eigs Remark}(i)). Step (v) requires at most $d$ eigenvector calculations at a cost of $O(\rho_{q^{d'}} d^6)$ each, and (vi) involves computing a $q$th power of an element of $K$ $O(d^4)$ times, each of which is $O(\rho_{q^{d'}} \log q)$, and so step (vi) is $O(\rho_{q^{d'}} d^4 \log q)$. Step (vii) costs less than (v) since we may complete it by considering the $1$-eigenspace. Combining these runtimes, Algorithm \ref{LabelEigenvaluesSymSquare(d'=d)} is $O(\rho_{q^{d'}} (d^7 + d^4 \log q))$.
\end{proof}
\subsubsection{The Case $d'=d-1$}
In this case, we have that $\ell_{d}=1$, and so the eigenvalues of $s$ in its action on $W_K$ are, by Lemma \ref{EigenstructureOnWKSymSquare},
\[
\{\ell_{ij}\mid 1\leq i\leq j\leq d\} = \{1\} \cup \{\ell_{id}=\ell_{i} \mid 1\leq i\leq d'\} \cup \{ \ell_{ij} \mid 1\leq i \leq j \leq d'\}.
\]
We now present the algorithm \texttt{LabelEigenvaluesSymSquare(d'=d-1)}, which applies to all of these cases. We proceed similary to the case $d'=d$ above, but this time we identify a suitable candidate for the orbit of $\ell_{1d}=\ell_1$ by noting that $\ell_{1d}^{q^{k-1}+1} = \ell_{1k}$ for $1 \leq k\leq d'$. 
\begin{algorithm}[ht]
\caption{LabelEigenvaluesSymSquare(d'=d-1)}\label{LabelEigenvaluesSymSquare(d'=d-1)}
\begin{flushleft}\textbf{Input:} A special element $s\in H$, and the eigenvalues of $s$ in its action on $W_K$, in the case that $d'=d-1$.\\
\textbf{Output:} A valid labelled set $\{\ell_{ij}\mid 1\leq i \leq j \leq d\}$ of eigenvalues, and a basis $\mathscr{F}_W=\{f_{ij}\mid (i,j)\in \App_W \}\cup \mathscr{F}'$ of $W_K$, satisfying the conditions in Lemma \ref{Final Lem SymSquare}.\\
\textbf{Procedure:}
\end{flushleft}
\begin{enumerate}
\item Compute the $q$th power of each eigenvalue, and sort them into ordered $\sigma$-orbits.
\item For each orbit $\Omega$ of eigenvalues, and choose an element $\alpha \in\Omega$ and label $\ell_{1d}:=\alpha$. For $1\leq k\leq d-1$, compute $\alpha^{q^{k-1}+1}$. If the result is an eigenvalue, label it $\ell_{1k}$; if not, then discard all labels, and choose another orbit $\Omega$. Once all of $\ell_{1k}$ have been labelled, proceed to (iii).
\item For $2\leq i\leq d'$, label $\ell_{id}=\ell_{1d}^{q^{i-1}}$; for $2\leq i\leq j\leq d'$, label $\ell_{ij} = \ell_{i-1,j-1}^q$.
\item For each $(1,j)\in\App_W$, set $f_{1j}$ to be the eigenvector of $\ell_{1j}$ having a $1$ in its first nonzero entry.
\item For $(i,j)\in\App_W$ with $i>1$, compute $f_{ij}$ using the $\sigma$-relations in Definition \ref{Sigma Relations SymSquare}.
\item If necessary, extend $\{f_{ij}\mid (i,j)\in\App_W\}$ to a basis for $W_K$ in any way.
\end{enumerate}
\end{algorithm}
\begin{proposition}\label{LabelEigenvaluesSymSquare(d'=d-1) Prop}
Algorithm \ref{LabelEigenvaluesSymSquare(d'=d-1)} (\texttt{LabelEigenvaluesSymSquare(d'=d-1)}) returns a valid labelling $\{\ell_{ij} \mid 1\leq i\leq j \leq d\}$ of the eigenvalues of $s$ on $W_K$ together with a basis $\mathscr{F}_W=\{ f_{ij} \mid (i,j)\in\App_W \}\cup \mathscr{F}'$ for $W_K$ satisfying the conditions in Lemma \ref{Final Lem SymSquare}; and has complexity 
\[
O(\rho_{q^{d'}} d^4(d^3+ \log q)),
\]
where $\rho_{q^{d'}}$ is the cost of a field operation in $K=\F{q^{d'}}$.
\end{proposition}
\begin{proof}
Setting $\ell_{i}= \ell_{id}$ for $1\leq i\leq d'$, and $\ell_d=1$, we have $\ell_{ij} = \ell_i \ell_j$ for $1\leq i\leq j\leq d$, and so the labelling $\ell_{ij}$ is valid. Since every orbit is tested, the procedure will eventually find an orbit (the true orbit of $\ell_{1d}$) satisfying this condition, and so always returns a valid labelling. Since for each $f_{ij}$ we have that $f_{ij}$ is an $\ell_{ij}$-eigenvector, and by the construction of $f_{ij}$ in (v) $\mathscr{F}_W$ satisfies the $\sigma$-relations in the Symmetric Square case for all $(i,j)\in\App_W$ (see Definition \ref{Sigma Relations SymSquare}), we have that $\mathscr{F}_W$ satisfies the conditions of Lemma \ref{Final Lem SymSquare}.\\\\
Step (i) has complexity $O(\rho_{q^{d'}} d^2\log q)$ (see Remark \ref{Labelling Eigs Remark}(i)), and after performing this step we can compute each power $\lambda^{q^{k-1}+1}$ with just one field multiplication in $K$. Thus we are guaranteed to find a suitable $\ell_{1d}$ after at most $d^2$ multiplications in $K$, and so step (ii) has complexity $O(\rho_{q^{d'}} d^2)$. Step (iii) is `free' since we have completed step (i) (again by Remark \ref{Labelling Eigs Remark}(i)).\\\\
Step (iv) requires at most $d$ eigenvector calculations at a cost of $O(\rho_{q^{d'}} d^6)$ each, and (v) involves computing a $q$th power of an element of $K$ $O(d^4)$ times, each of which is $O(\rho_{q^{d'}} \log q)$, and so the cost of step (v) is $O(\rho_{q^{d'}} d^4 \log q)$. Step (vi) costs less than step (iv), since we may complete it by a computation of the $1$-eigenspace. Combining these runtimes, the total cost of Algorithm \ref{LabelEigenvaluesSymSquare(d'=d-1)} is $O(\rho_{q^{d'}} (d^7+ d^4 \log q))$.
\end{proof}
\subsubsection{The Case $d'=d-2$}\label{LabelEigsd-2SymSquare Section}
In the case $d'=d-2$, we label $\ell_{d-1}= m_1, \ell_{d} = m_2$, where $m_i = \mu^{q^{i-1}}$, so the eigenvalues of $s$ in its action on $W_K$ are, by Lemma \ref{EigenstructureOnWKSymSquare} and since $m_1m_2 = \mu^{q+1} = 1$,
\[
\{1\} \cup \{\ell_{i} m_j \mid 1\leq i\leq d', 1\leq j\leq 2\} \cup \{ \ell_{ij} \mid 1\leq i \leq j \leq d'\} \cup \{m_1^2, m_2^2\}.
\]
We approach the problem by trying to identify those two $\sigma$-orbits of eigenvalues containing $\ell_1 m_1, \ell_1 m_2$ respectively. Our knowledge of $m_1^2$ makes the process of eliminating unsuitable candidates easy, since the orbit $\Omega$ of $\ell_1 m_1$ has the property that $\frac{m_2}{m_1} \Omega$ is itself the orbit of $\ell_1m_2$.
\begin{algorithm}[ht]
\caption{LabelEigenvaluesSymSquare(d'=d-2)}\label{LabelEigenvaluesSymSquare(d'=d-2)}
\begin{flushleft}\textbf{Input:} A special element $s\in H$, and the eigenvalues of $s$ in its action on $W_K$, in the case that $d'=d-2 \geq 6$.\\
\textbf{Output:} A valid labelled set $\{\ell_{ij}\mid 1\leq i \leq j \leq d\}$ of eigenvalues, and a basis $\mathscr{F}_W=\{f_{ij}\mid (i,j)\in \App_W \}\cup \mathscr{F}'$ of $W_K$, satisfying the conditions in Lemma \ref{Final Lem SymSquare}.\\
\textbf{Procedure:}
\end{flushleft}
\begin{enumerate}
\item Compute the $q$th power of each eigenvalue, and sort them into ordered $\sigma$-orbits.
\item There is exactly one orbit of length $2$, namely $\{m_1^2,m_2^2\}$: choose one eigenvalue from this orbit and label it $m_1^2$. Compute $\alpha = \frac{m_2}{m_1}$ as $(m_1^2)^{(q-1)/2}$ (recall $q$ is odd, so $(m_1^2)^{(q-1)/2} = m_1^{q-1} = m_2 /m_1$.
\item For each remaining orbit $\Omega$, choose $\beta\in \Omega$ and compute $\alpha\beta$. If this is an eigenvalue, label $\beta = \ell_1 m_1, \alpha\beta = \ell_1 m_2$ and proceed to (iv). If not, then try another orbit $\Omega$.
\item For $2\leq i \leq d'$, label $\ell_{i}m_2 = (\ell_{i-1} m_1)^q, \ell_{i}m_1 = (\ell_{i-1} m_2)^q$.
\item For $2\leq k \leq d'$, compute $\ell_{1k} = (\ell_{1}m_1) (\ell_{k} m_2)$. If this is not an eigenvalue, return to (iii) and choose another orbit $\Omega$.
\item For $(i,j)\in\App_W$ with $i\geq 2$, label $\ell_{ij} = \ell_{i-1,j-1}^q$.
\item For each $(1,j)\in\App_W$, set $f_{1j}$ to be the eigenvector of $\ell_{1j}$ having a $1$ in the first nonzero entry.
\item For $(i,j)\in\App_W$ with $i>1$, compute $f_{ij}$ using the $\sigma$-relations in Definition \ref{Sigma Relations SymSquare}.
\item If necessary, extend $\{f_{ij}\mid (i,j)\in\App_W\}$ to a basis for $W_K$ in any way.
\end{enumerate}
\end{algorithm}
\begin{proposition}\label{LabelEigenvaluesSymSquare(d'=d-2) Prop}
Algorithm \ref{LabelEigenvaluesSymSquare(d'=d-2)} (\texttt{LabelEigenvaluesSymSquare(d'=d-2)} a valid labelling $\{\ell_{ij} \mid 1\leq i\leq j \leq d\}$ of the eigenvalues of $s$ on $W_K$ together with a basis $\mathscr{F}_W=\{ f_{ij} \mid (i,j)\in\App_W\}\cup \mathscr{F}'$ for $W_K$ satisfying the conditions in Lemma \ref{Final Lem SymSquare}; and has complexity 
\[
O(\rho_{q^{d'}} d^4 ( d^3+ \log q)),
\]
where $\rho_{q^{d'}}$ is the cost of a field operation in $K$.
\end{proposition}
\begin{proof}
Setting $m_1$ as a square root of the chosen $m_1^2$ in step (ii), and $\ell_i = (\ell_im_1)/m_1$ for $1\leq i\leq d'$, $\ell_{d-1}= m_1, \ell_d = m_1^q$, we have for all $(i,j)$ that $\ell_{ij}=\ell_i\ell_j$, and this is a correct labelling of the eigenvalues, and since every orbit is tested in steps (iii)-(v), an orbit satisfying these properties is found, since the true orbit of $\ell_1m_1$ must eventually be tested. Since for all $\ell_{ij}$ we chose $f_{ij}$ to be an $\ell_{ij}$-eigenvector, and we construct $f_{ij}$ in step (viii) to satisfy the $\sigma$-relations for all $(i,j)\in \App_W$ (see Definition \ref{Sigma Relations SymSquare}), the basis $\mathscr{F}_W$ satisfies the conditions of Lemma \ref{Final Lem SymSquare}.\\\\
Step (i) involves $d^2$ $q$th-power calculations, and so costs $O(\rho_{q^{d'}}d^2\log q)$ (see Remark \ref{Labelling Eigs Remark}(i)). Step (ii) requires $2$ $q$th-power calculations (in the sense that we take powers bounded above by $q$), and so has complexity $O(\rho_{q^{d'}} \log q)$. Getting from step (iii) to the successful completion of step (v), in the worst case, requires the testing of $d$ orbits, and each test requires $d$ field multiplications -- note that by Remark \ref{Labelling Eigs Remark}(i), steps (iv), (vi) have zero cost -- hence steps (iii)-(vi) together have complexity $O(\rho_{q^{d'}} d^2)$.\\\\
Step (vii) requires at most $d$ eigenvector calculations at a cost of $O(\rho_{q^{d'}} d^6)$ each, and (viii) involves computing a $q$th power of an element of $K$ $O(d^4)$ times, each of which has complexity $O(\rho_{q^{d'}} \log q)$, and so step (viii) costs $O(\rho_{q^{d'}} d^4 \log q)$. Step (ix) costs less than step (vii), since we may complete it by computing the $1$-eigenspace of $s$. Combining these runtimes, Algorithm \ref{LabelEigenvaluesSymSquare(d'=d-2)} is $O(\rho_{q^{d'}} (d^7 + d^4 \log q))$.
\end{proof}
\subsection{Avoiding Division By Zero}\label{Division By Zero Section}
In the following steps of the algorithm there is a small chance that our procedure may attempt to divide by zero! To deal with this (very real) possibility we again use the techniques of randomised algorithms, and so we need to address two things: we must decide what to do when a division by zero is attempted, and we must bound the probability that the need will arise.
Should a division by zero be attempted during one of the \texttt{FindConstants} family of procedures, we simply observe that these procedures depend upon a random selection in the group $H$, and the division by zero is, in fact, dependent on the random choice made. Thus it is easily fixed by choosing another random element (of course, if this continues to occur we must return \texttt{FAIL}).\\\\
If a division by zero is attempted during one of the \texttt{FindPreimage} family of procedures, we must somehow `inject' randomness into proceedings: should $g\in G$, the input to \texttt{FindPreimage}, cause an error, we choose a random $h\in H$, and compute preimages under $\varphi$ of $h, gh^{-1}$. Then the preimage of $g$ is found by computing 
\[
\varphi^{-1}(g) = \varphi^{-1}(gh^{-1}) \varphi^{-1}(h),
\]
where here we use the notation $\varphi^{-1}(h)$ to mean a representative of the preimage: since this gives only a sign ambiguity, this is well-defined and gives the full preimage of $g$. We now describe precisely the conditions under which a division by zero may be attempted.
\begin{definition}\label{Divisibility SymSquare}
Let $K=\F{q^{d'}}$, and let $(a_{ij}) \in \GL(d,K)$. Then $(a_{ij})$ is said to have the \emph{divisibility property} if $a_{ij} \neq 0$ for all $(i,j)\in\App_W$.
\end{definition}
\begin{lemma}\label{Divisibility Equivalence Lem}
Let $K=\F{q^{d'}}$ for $q$ odd, let $(a_{ij})\in\GL(d,q^{d'})$, and for each $(i,j),(k,\ell)\in\App_W$, suppose that $\kappa_{ij,k\ell} = \frac{c_{ij}}{c_{k\ell}}(a_{ik}a_{j\ell} + (1-\delta_{ij})a_{i\ell}a_{jk})$, for $c_{ij},c_{k\ell}\neq 0$. Then $(a_{ij})$ has the divisibility property if and only if, for every $i,j,k$ with $(i,i), (i,j), (j,k),(k,k)\in\App_W$, we have $\kappa_{ii,jk}, \kappa_{ij,kk} \neq 0$.
\end{lemma}
\begin{proof}
This follows immediately from the Basic Equations, since $\kappa_{ii,jk} = \frac{c_{ii}}{c_{jk}} a_{ij}a_{ik}, \kappa_{ij,kk}=2 \frac{c_{ij}}{c_{kk}} a_{ik}a_{jk}$ and since $q$ is odd and all of the $c_{ij}$ are nonzero.
\end{proof}
In \cite[Lemma 4.8]{magaard2008recognition}, Magaard, O'Brien \& Seress managed to find a lower bound on the proportion of elements of an arbitrary subgroup $G \leqslant\GL(d,K)$ having the divisibility property in the Symmetric Square Case for $G\cong \SL(d,q)$: however, their argument depends entirely on the large order of $\SL(d,q)$, and hence cannot be applied to the other classical groups. We require a conjecture that a similar result holds:
\begin{conjecture}\label{Division By Zero Probability Symmetric Square Case}
Let $s \in G$ be a special element, and let $\{f_i\}$ be a basis of eigenvectors for $s_{V_K}$ as described in Lemma \ref{SpecialEigLem}. Let $g\in G$ be a random element of $G$, and let $(a_{ij})$ be the matrix of $g$ with respect to the basis $\mathscr{F}_V := \{f_i \mid 1\leq i\leq d\}$ of $V_K$. Then
\begin{enumerate}
\item For each $i,j$ we have $P(a_{ij} = 0)  < \frac{4}{q^d}$; and 
\item $P(a_{ij} \neq 0, \forall i,j) > \frac{5}{8}.$
\end{enumerate}
\end{conjecture}
When $G=\GL(d,q)$, this is precisely the statement of Lemma 4.8 in \cite{magaard2008recognition} (albeit with slightly different notation).
To computationally test Conjecture \ref{Division By Zero Probability Symmetric Square Case}, we construct a random conjugate of $G$ in $\GL(d,K)$, and choose a random sample of matrices from this random conjugate (in practice, we produce a random element $h\in \GL(d,K)$, choose random elements $\{g_i\}$ from a standard copy of $G$, and test their conjugates $\{g_i^h\}$.
We tested all groups $\SL(d,q), \SU(d,q), \Sp(d,q), \SO^{\epsilon}(d,q)$ for all relevant $d\leq 12, q\leq 13$: we tested $10$ random conjugates of the group in $\GL(d,K)$, and chose from each conjugate $100$ random elements. We found \emph{no} case of a matrix failing to possess the divisibility property.
Of course, it is easy to construct matrices which fail to possess the divisibility property: for example in the Symmetric Square case, most `nice' matrices, including the identity matrix, do not have the property. However, the sheer size of $\GL(d,K)$ means that a random conjugate of $G$ is unlikely to contain many `nice' matrices.
\subsection{Finding the Constants $c_{ij}$}\label{FindConstants SymSquare Section}
Having found, using the appropriate variant of \texttt{LabelEigenvalues} in Section \ref{LabelEigenvalues SymSquare Section} above, a basis $\mathscr{F}_W$ satisfying the conditions of Lemma \ref{Final Lem SymSquare}, we know (by the conclusions of this Lemma and Corollary \ref{Basic Equations Corollary SymSquare}) that there exist bases $\mathscr{F}_V,\mathscr{F}_V^-$ for $V_L$, and sets $\mathscr{C}=\{c_{ij} \mid (i,j)\in\App_W \},\mathscr{C}^-=\{c_{ij}^- \mid (i,j)\in\App_W  \}\subset L$, such that the action of $g\in G$ on $\mathscr{F}_V$ (or $\mathscr{F}_V^-$) can be calculated from the action on $\mathscr{F}_W$, so long as we know the values of certain $c_{ij}$ (or $c_{ij}^-$). This section is dedicated to the computation of these required constants.\\\\
Recall from Corollary \ref{Basic Equations Corollary SymSquare} that, for every $(i,j),(k,\ell)\in\App_W$, the \emph{Basic Equations in the Symmetric Square Case} hold:
\begin{equation}\tag{\ref{BasicEquationSymSquare}}
\kappa_{ij,k\ell} = \frac{c_{ij}}{c_{k\ell}}\left(a_{ik}a_{j\ell} + (1-\delta_{ij})a_{i\ell}a_{jk}\right)= \frac{c_{ij}^-}{c_{k\ell}^-}\left(a_{ik}^-a_{j\ell}^- + (1-\delta_{ij})a_{i\ell}^-a_{jk}^-\right).
\end{equation} 
where $a_{ij} = g_{f_if_j}, a_{ij}^- = g_{f_i^- f_j^-}, \kappa_{ij,k\ell} = g_{f_{ij}f_{k\ell}}$, and $\delta_{ij}=1$ when $i=j$ and $0$ otherwise. The first of these equations is the key to both the process of finding $c_{ij}$ (or $c_{ij}^-$), and later, finding the matrix $(a_{ij}) = (g_{f_if_j})$ for an arbitrary $g\in G$. However, in the course of our procedures, information is lost in the case that both sides of the equation are zero: this is addressed in Section \ref{Division By Zero Section}: recall from Definition \ref{Divisibility SymSquare} that we say a matrix $(a_{ij}) \in \GL(d,K)$ has the \emph{divisibility property} if $a_{ij}\neq 0$ for all $i,j$.
\begin{remark}\label{Appropriate Pairs}
Throughout this section and the next, we make frequent reference to Lemma \ref{AppLem SymSquare}, which states (in short) that, with a few exceptional cases, we have $\{(1,j)\mid 1\leq j\leq d, j\neq d'/2+1\} \subset \App_W$. We prove several results in this section which depend upon membership in $\App_W$, and so we do not technically require Lemma \ref{AppLem SymSquare} until we `use' the results to produce the Algorithms \texttt{FindConstantsSymSquare} and \texttt{FindPreimageSymSquare}. However, the reader should keep in mind that $(1,d'/2+1)$ is the only possible exception to the general rule that `$(1,j)$ is always in $\App_W$'.\\\\
Moreover, it is always true that $(i,j)\in\App_W$ whenever $(i-1,j-1) \in \App_W$.
\end{remark}
\subsubsection{Relations Between the Values $\kappa_{ij,k\ell}, c_{ij}, a_{ij}$}\label{RelationsSymSquare}
In this section we derive certain relations between the constants $\kappa_{ij,k\ell}$, $c_{ij}$, and $a_{ij}$, which are obtained through manipulations of (\ref{BasicEquationSymSquare}) along with the assumption that $(a_{ij})$ has the divisibility property. Note that while all of these relations apply to the `negative' versions of the $a_{ij}, c_{ij}$, we have no need for them.
\begin{lemma}\label{cijsquared} 
Suppose that $(1,1),(i,j)\in\App_W$. Suppose that $(a_{ij})$ has the divisibility property as in Definition \ref{Divisibility SymSquare}. If $i=j$, then $c_{ii}=1$, and if $i<j$ then
\[
c_{ij}^2 = \frac{\kappa_{ij,jj}^2}{4\kappa_{jj,jj}\kappa_{ii,jj}}.
\]
\end{lemma}
\begin{proof}
If $i=j$ the result follows immediately from Corollary \ref{Basic Equations Corollary SymSquare}. Suppose now that $i<j$. Then by (\ref{BasicEquationSymSquare}), noting that the $c_{ij}$ are, by definition, never zero, and since $\kappa_{ij,jj}, \kappa_{jj,jj}\neq 0$ since $g$ has the divisibility property (by Lemma \ref{Divisibility Equivalence Lem}), we have
\[ 
\frac{\kappa_{ij,jj}^2}{\kappa_{jj,jj}\kappa_{ii,jj}} = \frac{(\frac{c_{ij}}{c_{jj}}(2a_{ij}a_{jj}))^2}{\frac{c_{jj}}{c_{jj}}(a_{jj}a_{jj})\frac{c_{ii}}{c_{jj}}(a_{ij}a_{ij})}= 4c_{ij}^2,
\]
and the result follows. Note that since $q$ is odd, division by $4$ is well-defined.
\end{proof}
\begin{lemma}\label{formulae}
Suppose that $(1,1)\in\App_W$. Then the following hold for all pairwise distinct integers $i,j,\ell$ such that $(i,j), (j,\ell)\in\App_W$, when $(a_{ij})$ has the divisibility property as in Definition \ref{Divisibility SymSquare}:
\begin{enumerate}
\item $\kappa_{ii,ii} = a_{ii}^2$;
\item $c_{ij} = \displaystyle\frac{a_{ii}a_{ij}}{\kappa_{ii,ij}}$;
\item Define 
\[
(iii)_{i,j,\ell}:=\frac{\kappa_{ii,ij}\kappa_{ij,jj}}{2\kappa_{jj,jj}\kappa_{ii,jj}}\left(\kappa_{ij,j\ell} - \frac{\kappa_{ij,jj}\kappa_{jj,j\ell}}{2\kappa_{jj,jj}}\right).\]
Then $\displaystyle\frac{c_{ij}}{c_{j\ell}}a_{i\ell}a_{ii} = (iii)_{i,j,\ell}$.
\end{enumerate}
\end{lemma}
\begin{proof}
Part (i) follows on setting $i=j=k=\ell$ in (\ref{BasicEquationSymSquare}). For (ii), applying (\ref{BasicEquationSymSquare}) with pairs $(i,i),(i,j)$, we have
\[
\kappa_{ii,ij} = \frac{c_{ii}}{c_{ij}} (a_{ii}a_{ij}),
\]
and the result follows since $c_{ii} = 1$, by Lemma \ref{cijsquared}.\\\\
For (iii), note first that applying (\ref{BasicEquationSymSquare}), we have 
\begin{align*}
\displaystyle\frac{\kappa_{ii,ij}\kappa_{ij,jj}}{2\kappa_{jj,jj}\kappa_{ii,jj}} &=
\displaystyle\frac{(\frac{c_{ii}}{c_{ij}}a_{ii}a_{ij})(\frac{c_{ij}}{c_{jj}}2a_{ij}a_{jj})}{2(a_{jj}^2) (\frac{c_{ii}}{c_{jj}} a_{ij}^2)}\\
&= a_{ii}a_{jj}^{-1}.
\end{align*}
Secondly, again applying (\ref{BasicEquationSymSquare}),
\begin{align*}
\kappa_{ij,j\ell} - \displaystyle\frac{\kappa_{ij,jj}\kappa_{jj,j\ell}}{2\kappa_{jj,jj}} &=
\displaystyle\frac{c_{ij}}{c_{j\ell}} (a_{ij}a_{j\ell}+ a_{i\ell}a_{jj}) - \displaystyle\frac{(\frac{c_{ij}}{c_{jj}}2a_{ij}a_{jj})(\frac{c_{jj}}{c_{j\ell}} a_{jj}a_{j\ell})}{2a_{jj}^2}\\
&=
\displaystyle\frac{c_{ij}}{c_{j\ell}} (a_{ij}a_{j\ell}+ a_{i\ell}a_{jj}) - \displaystyle\frac{c_{ij}}{c_{j\ell}} a_{ij}a_{j\ell}\\
&=\displaystyle\frac{c_{ij}}{c_{j\ell}} a_{i\ell}a_{jj}.
\end{align*}
Multiplying the two gives the result.
\end{proof}
We now use Lemma \ref{formulae} to give a result which allows us to isolate the $a_{ij}$ from the $c_{ij}$ (we will use this to extract the $c_{ij}$ first, and once they are known it will be relatively easy to find the $a_{ij}$):
\begin{lemma}\label{a11a1kodd}
Let $(iii)_{i,j,k}$ be defined as in Lemma \ref{formulae}(iii), let $k >1$ be odd, and set $j=\frac{k+1}{2}$. Then if $(1,1),(1,j),(1,k),(j,k)\in\App_W$, and if $(a_{ij})$ satisfies the divisibility property as in Definition \ref{Divisibility SymSquare}, we have
\[
a_{1k}a_{11} =(iii)_{1,j,k} \left(\frac{\kappa_{1j,jj}^2}{4\kappa_{jj,jj}\kappa_{11,jj}}\right)^{\frac{q^{j-1}-1}{2}}.
\]
\end{lemma}
\begin{proof}
By Lemma \ref{formulae}(iii), we have 
\[
(iii)_{1,j,k}=\frac{c_{1j}}{c_{jk}}a_{1k}a_{11}.
\]
Now $c_{jk}=c_{1j}^{q^{j-1}}$, and hence $\frac{c_{1j}}{c_{jk}} = (c_{1j}^2)^{\frac{1-q^{j-1}}{2}}$. The result then follows by Lemma \ref{cijsquared}.
\end{proof}
It may seem that the result of Lemma \ref{a11a1kodd} is sufficient to determine the $a_{ij}$ for very many $(i,j)$ without any care for the values of the $c_{ij}$. However, for simplicity, and since things become increasingly complex when there are issues with $(i,j)\not\in\App_W$, we prefer to calculate the $c_{ij}$ in any case: it is better to deal with any potential difficulties in the preprocessing procedure \texttt{Initialise} rather than in the procedure \texttt{FindImage}, which may be run many times.
\begin{lemma}\label{c1kodd}
Let $(iii)_{i,j,k}$ be defined as in Lemma \ref{formulae}(iii), let $k >1$ be odd, and set $j=\frac{k+1}{2}$. If $(a_{ij})$ has the divisibility property as in Definition \ref{Divisibility SymSquare}, and $(1,1),(1,j),(1,k),(j,k)\in\App_W$, then we have
\[
c_{1k} =  \frac{(iii)_{1,j,k} }{\kappa_{11,1k}}\left(\displaystyle\frac{\kappa_{1j,jj}^2}{4\kappa_{jj,jj}\kappa_{11,jj}}\right)^{\frac{q^{j-1}-1}{2}}.
\]
\end{lemma}
\begin{proof}
The result follows by applying Lemmas \ref{formulae}(ii) and \ref{a11a1kodd} to $c_{1k}$.
\end{proof}
We now find analogous results to Lemmas \ref{a11a1kodd} and \ref{c1kodd} for even $k$.
\begin{lemma}\label{a11a1keven}
Suppose that $(1,1),(1,2)\in\App_W$, and suppose that $(a_{ij})$ satisfies the divisibility property as in Definition \ref{Divisibility SymSquare}. Let $c_{12}'$ be a square root of 
\[
\frac{\kappa_{12,22}^2}{4\kappa_{22,22}\kappa_{11,22}},
\]
and let $y = \frac{c_{12}}{c_{12}'}$. Then $y \in \{\pm 1\}$, and $c_{12}'\in K$. Moreover, for any even $k >2$, set $j=k/2$. Then if $(1,j),(1,k),(j,k-1)\in\App_W$, we have
\[
a_{1k}a_{11} = \left((iii)_{1,2,k} \frac{(c_{1,k-1})^q}{c_{12}'}\right)y.
\]
\end{lemma}
\begin{proof}
By Lemma \ref{cijsquared}, we have
\[
c_{12}^2 = \frac{\kappa_{12,22}^2}{4\kappa_{22,22}\kappa_{11,22}},
\]
and so since the square roots of the right hand side are $\pm c_{12}$, both square roots lie in $K$. Then labelling either of these $c_{12}'$, we have $c_{12} = \pm c_{12}'$, and so $y=\pm 1$.\\\\
For each $j \leq d'/2$ with $(1,j),(1,k),(j,k-1)\in\App_W$, set $k=2j$. By Lemma \ref{formulae}(iii),
\[
\frac{c_{12}}{c_{2k}}a_{1k}a_{11} = (iii)_{1,2,k},
\]
and since $c_{2k}=c_{1,k-1}^q$ we have 
\[
a_{1k}a_{11} = (iii)_{1,2,k} \frac{(c_{1,k-1})^q}{c_{12}} = \left((iii)_{1,2,k} \frac{(c_{1,k-1})^q}{c_{12}'}\right)y
\]
as required.
\end{proof}
\begin{lemma}\label{c1keven}
Let $k>2$ be even and suppose that $(1,1),(1,j),(1,k),(2,k)\in\App_W$, and $(a_{ij})$ has the divisibility property as in Definition \ref{Divisibility SymSquare}. Let $c_{12}',y$ be defined as in Lemma \ref{a11a1keven} above, and define
\[
c_{1k}'= \frac{(iii)_{1,2,k} }{\kappa_{11,1k}} \frac{(c_{1,k-1})^q}{c_{12}'}.
\]
Then $c_{1k}=c_{1k}'y$.
\end{lemma}
\begin{proof}
Using Lemmas \ref{formulae}(ii) and \ref{a11a1keven}, we have
\begin{align*}
c_{1k}
&= \displaystyle\frac{1}{\kappa_{11,1k}}a_{11}a_{1k}
= \displaystyle\frac{1}{\kappa_{11,1k}} \left((iii)_{1,2,k} \displaystyle\frac{c_{1,k-1}^q}{c_{12}'}\right)y,
\end{align*}
and the result follows.
\end{proof}
The following result proves that if we `incorrectly guessed' the value of $c_{12}$ (that is, if $y=-1$), then we will find instead the values $c_{ij}^-$, and so without loss of generality we may assume that $y=1$.
\begin{lemma}\label{SignAmbiguityLem}
Let $\mathscr{C} = \{ c_{ij} \mid (i,j)\in\App_W\}, \mathscr{C}^-= \{ c_{ij}^- \mid (i,j)\in\App_W\}$ be as defined in Corollary \ref{Basic Equations Corollary SymSquare}. Let $y=\pm 1$, and define $\mathscr{C}':=\{c_{ij}'\mid (i,j)\in\App_W\}$ as follows. For $1\leq j\leq d'$ with $(1,j)\in \App_W$, let 
\[
c_{1j}'=
\begin{cases}
c_{1j}  & \text{ if $j$ is odd; and}\\
c_{1j}y & \text{ if $j$ is even;}
\end{cases}
\]
and for $2\leq i\leq j\leq d'$ with $(i,j)\in\App_W$, set $c_{ij}' = (c_{i-1,j-1}')^q$. For $(i,j)\in\App_W$ with $j>d'$, set $c_{ij}'=c_{ij}$.\\\\
Then for all $(i,j)\in\App_W, 1\leq i\leq j\leq d'$, we have that
\[
c_{ij}'=
\begin{cases}
c_{ij}y &\text{ if $1\leq i\leq j\leq d'$ and $j-i$ is odd; and} \\
c_{ij} &\text{ otherwise.}
\end{cases}
\]
In particular, we have
\[
\mathscr{C}'=
\begin{cases}
\mathscr{C} &\text{ if $y=1$; and} \\
\mathscr{C}^{-} &\text{ if $y=-1$.}
\end{cases}
\]
\end{lemma}
\begin{proof}
If $i=1$, then the result follows immediately from the definition of $c_{ij}'$. We now suppose that $i>1$ and proceed by induction. If $j-i$ is odd, we have that 
\[
c_{ij}' = (c_{i-1,j-1}')^q = (c_{i-1,j-1})^q  = c_{ij}.
\]
If $j-i$ is even then 
\[
c_{ij}' = (c_{i-1,j-1}')^q = (c_{i-1,j-1} y)^q = c_{ij} y,
\]
since $y^q=y$.\\\\
The second assertion follows trivially when $y=1$, and when $y=-1$ we have that the definition of $c_{ij}^{-} \in \mathscr{C}^{-}$ matches exactly the definition of $c_{ij}'\in\mathscr{C}'$ when $y=-1$. 
\end{proof}
{
\begin{algorithm}
\caption{FindConstantsSymSquare}\label{FindConstantsSymSquare}
\begin{flushleft}
\textbf{Input:} A basis $\mathscr{F}_W=\{f_{ij}\mid (i,j)\in \App_W \}\cup \mathscr{F}'$ of $W_K$, satisfying the conditions in Lemma \ref{Final Lem SymSquare}, and an acceptable probability of failure $\epsilon \in (0,1)$.\\
\textbf{Output:} One of the sets $\{c_{ij}\mid 1\leq i\leq j\leq d'\},\{c_{ij}^-\mid 1\leq i\leq j\leq d'\}$ as described in Corollary \ref{Basic Equations Corollary SymSquare}; or \texttt{FAIL}.\\
\textbf{Procedure:}
\end{flushleft}
\begin{enumerate}
\item Choose a random element $g\in G$, and find the matrix $(\kappa_{ij,k\ell})$ of $g$ with respect to the basis $\mathscr{F}_W$. If at any time during the rest of the procedure a division by zero is attempted, choose another random element and begin again, until $T$ selections have been made, where $T=\ceiling{4 {\log \epsilon^{-1}}}$. If the steps below cannot be completed on any of there $T$ random elements then return \texttt{FAIL}.
\item Set $c_{11}=1$, and for $2\leq j < (d'+1)/2$, set $k = 2j - 1$, and compute
\[
c_{1k} =  \frac{(iii)_{1,j,k} }{\kappa_{11,1k}}\left(\frac{\kappa_{1j,jj}^2}{4\kappa_{jj,jj}\kappa_{11,jj}}\right)^{\frac{q^{j-1}-1}{2}},
\]
where $(iii)_{1,j,k}$ is as in Lemma \ref{formulae}(iii). This provides $c_{1k}$ for all $(1,k)\in\App_W$ with $k$ odd.
\item Compute $c_{12}$ as one of the square roots of 
\[
\frac{\kappa_{1222}^2}{4\kappa_{2222}\kappa_{1122}}.
\]
\item For $2\leq j \leq d'/2$, set $k=2j$. If $(1,k),(2,k)\in\App_W$, then compute
\[
c_{1k}= \frac{(iii)_{1,2,k} }{\kappa_{11,1k}} \frac{(c_{1,k-1})^q}{c_{12}}.
\]
\item If there exists $k$ such that $(1,k)\in\App_W$ and we have not yet computed $c_{1k}$, then $d',d'/2$ are even and we compute
\[
c_{1k} = (iii)_{13k} \frac{(c_{1,k-2})^{q^{2}}}{c_{13}\kappa_{11,1k}}.
\]
\item For $(i,j)\in\App_W$ with $2\leq i\leq j\leq d$, compute $c_{ij} = c_{i-1,j-1}^q$.
\end{enumerate}
\end{algorithm}
}
\begin{proposition}\label{FindConstantsSymSquare Prop}
Assume that Conjecture \ref{Division By Zero Probability Symmetric Square Case} holds. Then if $G \not\in \{\Sp(d,3), \SO^{\epsilon}(d,3)\}$, then Algorithm \ref{FindConstantsSymSquare} returns returns correct values $c_{ij} = c_{ij}^{\pm}$ for $1\leq i\leq j\leq d'$, or \texttt{FAIL}; and is a Las Vegas algorithm with complexity 
\[
O((\xi_{H} + \rho_{q^{d'}} d(d^5 + \log q))\log \epsilon^{-1}),
\] 
where $\xi_H$ is the cost of choosing random elements from $H$, and $\rho_{q^{d'}}$ is the cost of a field operation in $K$.
\end{proposition}
\begin{proof}
Subject to $(a_{ij})$ having the divisibility property, the correctness of the values $c_{ij}$ follows from Lemmas \ref{c1kodd}, \ref{c1keven}, \ref{SignAmbiguityLem}, and \ref{AppLem SymSquare}, noting that we may return the set $\{c_{ij}^-\}$ in place of $\{c_{ij}\}$.\\\\
For each randomly selected $g$, we compute the matrix $(\kappa_{ij,k\ell})$: this requires matrix conjugation: conjugating by a matrix is equivalent to one inversion operation and two multiplications, for a total cost of $O(\rho_{q^{d'}} n^3) = O(\rho_{q^{d'}} d^6)$.\\\\
The computation of $(iii)_{1jk}$ requires a constant number of field operations in $K$, and so has cost $O(\rho_{q^{d'}})$. Step (ii)'s most expensive operation is the computation of a power of order $O(q^{d})$, and so its cost is $O(\rho_{q^{d'}} \log q^{d}) = O(\rho_{q^{d'}} d\log q)$. Step (iii) requires the computation of a square root in $K$, which has cost $O(\rho_{q^{d'}} d \log q)$. Steps (iv)-(vi) require only the computation of $q$th powers, and so have complexity less than step (ii). Thus the procedure costs $O(\rho_{q^{d'}}(d^6 + d\log q))$ if $(a_{ij})$ has the divisibility property.\\\\
Assuming Conjecture \ref{Division By Zero Probability Symmetric Square Case} holds, we have that the probability that $(a_{ij})$ has the divisibility property is at least $1/2$, and so setting $T = \ceiling{4\log\epsilon^{-1}}$, the procedure returns \texttt{FAIL} with probability less than $\epsilon$, and has complexity $O(T(\xi_H +\rho_{q^{d'}}(d^6 + d\log q)))= O((\xi_{H} + \rho_{q^{d'}} d^5(d+\log q))\log\epsilon^{-1})$.
\end{proof}
\section{Probability and Proportions}
The effectiveness of any algorithm which chooses random elements is dependent upon probability: namely, the probability that a randomly chosen element has the properties we need. In our case there are two issues at hand: the probability that a randomly chosen element has the required eigenstructure, and later that randomly chosen element do not have zeroes in places that we don't want.
\subsection{Counting Special Elements}\label{Section:Quokka}
In this Section we determine, using the Quokka Theory of Niemeyer and Praeger, lower bounds for the proportions of Special Elements in Classical Groups. We provide only a very brief summary of Quokka Theory here: for more see \cite{lubeck2009finding,niemeyer2010estimating}.\\\\

Quokka sets are subsets $Q$ of a finite group $G$ of Lie Type whose proportion in $G$ can be found by determining certain proportions in maximal tori in $G$ and in the Weyl group of $G$ (respectively, an abelian group and a permutation group -- both much simpler to deal with). Recall that each element $g\in G$ has a unique Jordan decomposition $g =  su$, where $s\in G$ is semisimple, $u\in G$ is unipotent and $su=us$ (with $s$ and $u$ called the \emph{semisimple part} and $u$ the \emph{unipotent part} of $g$ respectively \cite[p.~11]{carter1993finite}). \\
A nonempty subset $Q$ of $G=\GL(n,q)$ is a \emph{quokka set} if the following two conditions hold:
\begin{itemize}
\item[(i)] if $g\in G$ has Jordan decomposition $g=su$ with semisimple part $s$ and unipotent part $u$, then $g\in Q$ if and only if $s \in Q$;
\item[(ii)] $Q$ is a union of $G$-conjugacy classes.
\end{itemize}
By \cite[Lemma XX]{NISets}, the characteristic polynomial of the semisimple part $s$ of $g$ is the same as the characteristic polynomial of $g$: thus all properties of the eigenvalues of a group element are preserved in this `transition' to $s$. It follows that:
\begin{lemma}\label{SpecialQuokka}
Let $S$ be the set of Special Elements of a finite group of Lie Type $G$ as in Definition \ref{SpecialDefn}. Then $S$ is a Quokka set.
\end{lemma}
Suppose that $\bar{\F{q}}$ is the algebraic closure of $\F{q}$, with $F$ the Frobenius morphism (so that the fixed points of $F$ in $\bar{\F{q}}$ are precisely $\F{q}$). Then as outlined in~\cite[Section 3]{lubeck2009finding}, choose a maximal torus $T_0$ of $ \GL(n,\bar{\F{q}})$ so that $W = N_{\hat{G}}(T_0)/T_0$ is the corresponding Weyl group (isomorphic to a subgroup of $S_d$).\\\\
A subgroup $H$ of the connected reductive algebraic group $\GL(n,\bar{\F{q}})$ is said to be \emph{$F$-stable} if $F(H)=H$, and for each such subgroup $H$, we write $H^F=H\cap \GL(n,\F{q})$. 
We define an equivalence relation on $W$ as follows: elements $w,w'\in W$ are \emph{$F$-conjugate} if there exists $x\in W$ such that  $w' =x^{-1} w x^F$. 
The equivalence classes of this relation on $W$ are called \emph{$F$-conjugacy classes} \cite[p.~84]{carter1993finite}. 
The $\GL(n,\F{q})$-conjugacy classes of $F$-stable maximal tori are in one-to-one correspondence with the $F$-conjugacy classes of the Weyl group $W$. 
The explicit correspondence is given in~\cite[Proposition 3.3.3]{carter1993finite}.

Let $\mathcal{C}$ be the set of $F$-conjugacy classes in $W$ and, for each $C\in\mathcal{C}$, let $T_C$ be a representative element of the family of $F$-stable maximal tori corresponding to $C$. 
The following theorem is a direct consequence of \cite[Theorem 1.3]{niemeyer2010estimating}.
\begin{theorem}\label{the:quokka}
Suppose that $Q \subseteq G=\GL(n,q)$ is a quokka set. Then
\begin{equation}\label{QuokkaEqn}
\frac{|Q|}{|G|} =  \sum_{C \in \mathcal{C}}\frac{|C|}{|W|} \frac{|T_C^F\cap  Q|}{|T_C^F|}.
\end{equation}
where $W$ is the Weyl group of $G$, $\mathscr{C}$ is the set of conjugacy classes of $F$-stable maximal tori in $G$, and $T_C$ is a representative torus in the conjugacy class $C$. 
\end{theorem}
\afterpage{
\begin{landscape}
\quad\newline\newline\quad\newline\newline\quad\newline\newline
\centering{
\renewcommand{\arraystretch}{2}
\begin{tabular}{|c|c|c|c|c|c|c|c|c|c|c|}
\hline
\hline
Case 		& 	$d$ &	$G$ 				& $F$ 				 & $G^F$		  & $W$ 			 & Structure of Torus $T_C$  & 
$\dfrac{|C|}{|W|}$ & lb for $\dfrac{|S\cap T_C|}{|T_C|}$ & lb for $\dfrac{|S|}{|G|}$ & Condition \\
\hline
\hline
$A_{r}$ 	& $r+1$ & $\GL(d, \FBar{q})$ 	& ${\sigma_q}$		 & $\GL(r+1,q)$	  & $S_{r+1}$		 &  $\Z{\frac{q^d-1}{q-1}\gcd(d,q-1)}$ & 
$\dfrac{1}{d}$ &  $\dfrac{1}{3d\log q}$
& $\dfrac{1}{3d^2\log q}$ & \\
\hline
$^{2\!\!}A_{r}$ & $r+1$ 	& $\SU(d, \FBar{q})$& ${\sigma_{q}}(-T)$ & $\SU(r+1,q)$ &  $S_{r+1}$		 & 	$\Z{\frac{q^d+1}{q+1}\gcd(d,q+1)}$ & 
$\dfrac{1}{d}$ &  $\dfrac{1}{4d\log q}$  &  $\dfrac{1}{4d^2\log q}$ & $d$ odd\\

 	&  &   &   &   	 & 	  & $\Z{q^{d-1}+1}$ &
$\dfrac{1}{d-1}$ &  $\dfrac{1}{3d\log q}$ &  $\dfrac{1}{3d^2\log q}$ & $d$ even\\

\hline
$B_{r}$ 	& $2r+1$ & $\SO(d, \FBar{q})$ & ${\sigma_{q}}$	 & $\SO(2r+1,q)$	  & $S_{2}\wr S_r$	 
&	$\Z{q^{(d-1)/2}+1}$ &
$\dfrac{1}{d-1}$ &  $\dfrac{1}{3d\log q}$  &  $\dfrac{1}{3d^2\log q}$ & \\
\hline
$C_{r}$ & $2r$ 	& $\Sp(d, \FBar{q})$ & ${\sigma_{q}}$	 & $\Sp(2r,q)$	  & $S_{2}\wr S_r$	 
&	$\Z{q^{d/2}+1}$	& 
$\dfrac{1}{d}$ & $\dfrac{1}{3d\log q}$ & $\dfrac{1}{3d^2\log q}$ & \\
\hline
$D_{r}$ & $2r$ 	& $\SO^+(d, \FBar{q})$ & ${\sigma_{q}}$	 & $\SO^+(2r,q)$	  & $(S_{2}\wr S_r)\cap A_{2r}$	 
&	$\Z{q^{(d-2)/2}+1} \times \Z{q+1}$ &
$\dfrac{1}{d-2}$ & $\dfrac{1}{9d\log^2 q}$ & $\dfrac{1}{9d^2\log^2 q}$ & \\
\hline
$^{2\!}D_{r}$ & $2r$ 	& $\SO^-(d, \FBar{q})$ & ${\sigma_{q}}(-T)$	 & $\SO^-(2r,q)$	  & $(S_{2}\wr S_r)\cap A_{2r}$	 
&	$\Z{q^{d/2}+1}$ &
$\dfrac{2}{d}$  & $\dfrac{1}{3d\log q}$ & $\dfrac{2}{3d^2\log q}$ & \\
\hline
\hline
\end{tabular}
\captionof{table}{Weyl Group and Torus Structure for Counting Special Elements in Classical Groups}
\label{TableOfWeylGroups}
}
\end{landscape}
}
\begin{theorem}\label{ProportionsThm}
Let $G^F$ be as in the $5$th column of Table \ref{SpecialTable},let $d$ be the natural dimension of $G^F$, and let $S$ be the set of special elements of $G$, as defined in Definition \ref{SpecialDefn} (see Table \ref{SpecialTable} for reference). Then the proportion $|S|/|G^F|$ is bounded below by the corresponding value in the final column of Table \ref{SpecialTable}.
\end{theorem}
In all cases, the proof is similar to the proof of \cite[Theorem 1.9]{niemeyer2010estimating}, in which Niemeyer \& Praeger use this theory to determine the proportion of $\ppd(d;q;d')$-elements. In particular, the arguments used there for finding the proportion $\frac{|C|}{|W|}$ are nearly identical: the Weyl Group and proportion are as listed in Table \ref{TableOfWeylGroups}. Thus we summarise the proof and give the important values in Table \ref{TableOfWeylGroups}: for a very detailed proof in all cases, see \cite[Chapter 6]{corr2013estimation}. 
The maximal torus corresponding to the $F$-conjugacy class of $W$ is given in the $7$th column of Table \ref{TableOfWeylGroups}, and is exactly as in the proof of \cite[Theorem 1.9]{niemeyer2010estimating}: the only difference in our case is the extra conditions we impose on the order of a special element: this difference manifests in the proportion $|S \cap T_C|/||T_C|$: for this we need the following Lemma.
\begin{lemma}\label{ProportionsInCyclicGroupsI}
Let $a$ be a positive integer, let $G=\Z{a}$, and let $X$ be the set of elements of $G$ of order $a$. Then $|X| = \varphi(a) > \frac{a}{3\log a}$, where $\varphi$ is Euler's Totient Function, which counts the number of positive integers strictly less than $a$ which are coprime to $a$.
\end{lemma}
\begin{proof}
Let $t$ be a divisor of $a$. The proportion of elements of $\Z{a}$ having order exactly $t$ is $\varphi(t)$, and the second inequality (which we have simplified from $\varphi(a) > \left(\frac{\log 2}{2}\right) \left(\frac{a}{\log a}\right)$ since $\log 2 > \frac{2}{3}$) can be found in \cite[I.1.5]{mitrinovic1996handbook} (citing \cite{hatalova1970remarks}).
\end{proof}
Thus in every case, we apply Lemma \ref{ProportionsInCyclicGroupsI} to bound below the proportion of elements of the torus $T_C$ contained in $S$ (that is, having the required order as in Table \ref{SpecialTable}), to bound below the value $|S \cap T_C|/|T_C|$, with the results given in the $9$th column of Table \ref{TableOfWeylGroups} (in the case $d'=d-2$, we must apply Lemma \ref{ProportionsInCyclicGroupsI} twice: once to each cyclic component). We obtain a lower bound for $|S|/|G|$ ($10$th column) by combining this with the proportion $|C|/|W|$ ($8$th column) found in \cite{niemeyer2010estimating, corr2013estimation}.
\section{Implementations}
The original (unofficially released) implementation of this algorithm was in GAP and at the time of writing has been made public at the first author's website, or by direct contact. It is likely to be implemented in a more official way in the future. A MAGMA implementation of this and other functions (namely, similar algorithms for all absolutely irreducible modules of degree at most $d^2$) is also in development, and the Symmetric Square case is complete.
We perform tests using the implementation in MAGMA, comparing the runtimes (in seconds) in the Linear case against MAGMA's algorithm \texttt{RecogniseSL}, and in the Symplectic case against MAGMA's \texttt{RecogniseSpOdd}. Both are implementations of the Kantor-Seress Black Box algorithm \cite{kantor2001black}. Note that while in the Linear case, the Magaard-O'Brien-Seress algorithm \cite{magaard2008recognition} is implemented in GAP, this is essentially identical to our algorithm, and so we compare against the Black Box algorithm to illustrate the effectiveness of our methods. In practice, our implementation is slightly more efficient than the existing implementation, and both are considerably faster than the Black Box (of course, the Black Box methods have a much wider scope).\\\\
The MAGMA implementation was produced in July 2013, with the help and hospitality of the University of Auckland (in particular, Professor Eamonn O'Brien).\\\\
\begin{table}[t]
\begin{center}%
\begin{tabular}{||c||c|c||c|c||}
\hline
\hline
 \multicolumn{5}{||c||}{$\SL(10,q) \leqslant \GL(55,q)$}\\
\hline
 {} & \multicolumn{2}{c||}{Black Box} & \multicolumn{2}{c||}{New} \\
\hline
\hline
{$q$}  & \texttt{Init} & \texttt{Preimage}  & \texttt{Init} & \texttt{Preimage} \\
\hline
\hline
4	&	2.402	&	0.13354	&	0.07025	&	0.00815	\\
9	&	33.727	&	2.13191	&	1.33375	&	0.171835	\\
37	&	232.847	&	8.57023	&	1.95	&	0.0902475	\\
\hline
\hline
 \multicolumn{5}{||c||}{$\Sp(10,q) \leqslant \GL(55,q)$}\\
\hline
 {} & \multicolumn{2}{c||}{Black Box} & \multicolumn{2}{c||}{New} \\
\hline
\hline
{$q$}  & \texttt{Init} & \texttt{Preimage}  & \texttt{Init} & \texttt{Preimage} \\
\hline
\hline
									
5	&	3.338	&	0.22542	&	0.429	&	0.042745	\\
9	&	30.904	&	1.60463	&	1.174	&	0.0689925	\\
37	&	150.806	&	5.46581	&	1.478	&	0.1212525	\\

\hline
\hline
\end{tabular}
\caption{Comparison of runtimes (in seconds) between Black Box and specialised algorithms in MAGMA.}\label{TableComparisonRuntimes}

\end{center}
\end{table}
\begin{table}[t]
\begin{center}%
\begin{tabular}{||c||c|c||c|c||c|c||}
\hline
 \hline
 {$q$} & \multicolumn{2}{c||}{$4$} & \multicolumn{2}{c||}{$5$} & \multicolumn{2}{c||}{$8$} \\
\hline
\hline
 \multicolumn{7}{||c||}{$\SL(d,q)$ }\\
\hline
{$d$}  & \texttt{Init} & \texttt{Preimage}  & \texttt{Init} & \texttt{Preimage}  & \texttt{Init} & \texttt{Preimage} \\
\hline
\hline

10	&	0.1405	&	1.63	&	0.9205	&	7.839	&	0.398	&	1.607	\\
14	&	1.5365	&	4.493	&	6.872	&	95.527	&	5	&	7.4645	\\
20	&	17.9245	&	19.227	&	89.1155	&	922.863	&	17.8075	&	24.2035	\\

\hline
\hline
 \multicolumn{7}{||c||}{$\Sp(d,q)$ }\\
\hline
{$d$}  & \texttt{Init} & \texttt{Preimage}  & \texttt{Init} & \texttt{Preimage}  & \texttt{Init} & \texttt{Preimage} \\
\hline
\hline
10 & 0.117	& 0.011465	&	0.67475	&   0.071	    &	0.503	& 0.0197725 \\

14 & 1.84475	& 0.0457875	&	6.279	&   0.941465	&	3.8845	& 0.0523775 \\

20 & 18.2205	& 0.183105	&	123.12775 &	9.5988575	&	20.0345 &	0.2331425 \\

\hline
\hline
 \multicolumn{7}{||c||}{$\SU(d,q)$ -- $d$ even}\\
\hline
{$d$}  & \texttt{Init} & \texttt{Preimage}  & \texttt{Init} & \texttt{Preimage}  & \texttt{Init} & \texttt{Preimage} \\
\hline
\hline
10   &    0.823   &    	0.061035	   &    	2.05525   &    	0.12909	   &    	1.3805   &    	0.0757375 \\
14   &    5.9865   &    	0.0974225	   &    	23.18175   &    	2.5565	   &    	4.368   &    	0.0771025\\
20   &    25.1825	   &    0.2919175	   &    	315.3325	   &    28.06579	   &    	47.56475   &    	0.556885\\

\hline
\hline
 \multicolumn{7}{||c||}{$\SU(d,q)$ -- $d$ odd}\\
\hline
{$d$}  & \texttt{Init} & \texttt{Preimage}  & \texttt{Init} & \texttt{Preimage}  & \texttt{Init} & \texttt{Preimage} \\
\hline
\hline

9   &    0.542   &    	0.0047575	   &    	3.15	   &    0.009985	   &    	0.9905	   &    0.007565\\
13   &    1.68875   &    	0.0157175	   &    	22.5265   &    	0.30116		   &    3.7635	   &    0.023165\\
19   &    45.5795   &    	0.1873175	   &    	227.84325	   &    10.75861	   &    	76.5225   &    0.22316\\

\hline
\hline
\multicolumn{7}{||c||}{ $ $}\\
\hline
\hline
 {$q$} & \multicolumn{2}{c||}{$5$} & \multicolumn{2}{c||}{$7$} & \multicolumn{2}{c||}{$9$} \\
\hline
\hline
 \multicolumn{7}{||c||}{$\SO^-(d,q)$}\\
\hline
{$d$}  & \texttt{Init} & \texttt{Preimage}  & \texttt{Init} & \texttt{Preimage}  & \texttt{Init} & \texttt{Preimage} \\
\hline
\hline
10 &  	0.706 & 	0.06244		&	0.75275	&	0.0753875	&	1.02575	&	0.0678225	\\

14 &	7.176	 & 0.945755		&	
7.96775	&	1.265245	&	21.99625	&	2.5359525	\\
20 &	126.275 & 	9.5132925	&	
135.2685	&	12.0210075	&	177.7475	&	15.30327	\\
\hline
\hline
 \multicolumn{7}{||c||}{$\SO^+(d,q)$}\\
\hline
{$d$}  & \texttt{Init} & \texttt{Preimage}  & \texttt{Init} & \texttt{Preimage}  & \texttt{Init} & \texttt{Preimage} \\
\hline
\hline
10    & 0.10925	    &    0.0102175	    & 0.90875	&	0.0780775	&	1.2285	&	0.077415	\\

14    & 7.141	    &    0.83285	    & 
0.513	&	0.8535975	&	14.1765	&	2.0481375	\\
20    & 96.54125    &    	9.5089875	& 
112.348	&	9.8805575	&	150.4045	&	13.3346925	\\

\hline
\hline
 \multicolumn{7}{||c||}{$\SO^0(d,q)$}\\
\hline
{$d$}  & \texttt{Init} & \texttt{Preimage}  & \texttt{Init} & \texttt{Preimage}  & \texttt{Init} & \texttt{Preimage} \\
\hline
\hline
9 & 	0.113	&		0.01439	& 
0.35875	&	0.034555	&	0.57325	&	0.0407175	\\
13 &	2.258	&		0.21146	& 
2.453	&	0.2374325	&	4.18075	&	0.3691775	\\
19 &	61.875	&		8.86234	& 
55.31025	&	7.1629025	&	108.90825	&	10.0354275	\\
\hline
\hline
\end{tabular}
\caption{Average runtimes (in seconds) for various groups (in MAGMA).}\label{TableRuntimesSym}
\end{center}
\end{table}
We provide this comparison (see Table \ref{TableComparisonRuntimes}) in the Linear and Symplectic cases, and the runtime gains in all cases are comparable. Table \ref{TableRuntimesSym} gives sample runtimes for all cases for various values of $d$ and $q$. The stated runtimes are averaged over several runs of \texttt{Initialise} and several hundred runs of \texttt{FindPreimage}. All times are given in seconds.


\begin{thebibliography}{10}

\bibitem{beals2005constructive}
Robert Beals, Charles~R Leedham-Green, Alice~C Niemeyer, Cheryl~E Praeger, and
  {\'A}kos Seress.
\newblock Constructive recognition of finite alternating and symmetric groups
  acting as matrix groups on their natural permutation modules.
\newblock {\em Journal of Algebra}, 292(1):4--46, 2005.

\bibitem{carter1993finite}
Roger~W. Carter.
\newblock {\em Finite groups of Lie type: Conjugacy classes and complex
  characters}.
\newblock Wiley, 1993.

\bibitem{corr2013estimation}
Brian Corr.
\newblock {\em Estimation and Computation with Matrices Over Finite Fields}.
\newblock PhD thesis, PhD thesis, University of Western Australia, 2013.

\bibitem{NISets}
Brian~P. Corr, Tomasz Popiel, and Cheryl~E. Praeger.
\newblock Nilpotent-independent sets and counting in matrix algebras.
\newblock 2014.
\newblock In Preparation.

\bibitem{hartleyhawkes}
B.~Hartley and T.~O. Hawkes.
\newblock {\em Rings, modules and linear algebra}.
\newblock Chapman \& Hall, London, 1980.

\bibitem{hatalova1970remarks}
H~Hatalov{\'a} and T~{\v{S}}al{\'a}t.
\newblock Remarks on two results in the elementary theory of numbers.
\newblock {\em Acta Fac. Rerum Natur. Univ. Comenian. Math}, 20:113--117, 1970.

\bibitem{holt1996testing}
Derek~F Holt, Charles~R Leedham-Green, EA~O'Brien, and Sarah Rees.
\newblock Testing matrix groups for primitivity.
\newblock {\em Journal of Algebra}, 184(3):795--817, 1996.

\bibitem{kaltofen1998subquadratic}
Erich Kaltofen and Victor Shoup.
\newblock Subquadratic-time factoring of polynomials over finite fields.
\newblock {\em Mathematics of Computation of the American Mathematical
  Society}, 67(223):1179--1197, 1998.

\bibitem{kantor2001black}
William~M Kantor and {\'A}kos Seress.
\newblock {\em Black box classical groups}.
\newblock Number 708 in Memoirs of the American Mathematical Society. American
  Mathematical Society, 2001.

\bibitem{knuth1997volume}
Donald~E Knuth.
\newblock Volume 2: Seminumerical algorithms.
\newblock {\em The Art of Computer Programming}, page 192, 1997.

\bibitem{leedham2001computational}
Charles~R Leedham-Green.
\newblock The computational matrix group project.
\newblock {\em Groups and computation III}, 8:229--247, 2001.

\bibitem{LidlNiederreiter}
Rudolf Lidl and Harald Niederreiter.
\newblock {\em Finite fields}, volume~20 of {\em Encyclopedia of Mathematics
  and its Applications}.
\newblock Cambridge University Press, Cambridge, second edition, 1997.
\newblock With a foreword by P. M. Cohn.

\bibitem{lubeck2009finding}
Frank L{\"u}beck, Alice~C Niemeyer, and Cheryl~E Praeger.
\newblock Finding involutions in finite lie type groups of odd characteristic.
\newblock {\em Journal of Algebra}, 321(11):3397--3417, 2009.

\bibitem{magaard2008recognition}
Kay Magaard, EA~O'Brien, and {\'A}kos Seress.
\newblock Recognition of small dimensional representations of general linear
  groups.
\newblock {\em J. Aust. Math. Soc}, 85(2):229--250, 2008.

\bibitem{mitrinovic1996handbook}
DS~Mitrinovic, J~S{\'a}ndor, and B~Crstici.
\newblock {\em Handbook of Number Theory, Mathematics and Its Applications
  351}.
\newblock Kluwer Academic Publishers, 1996.

\bibitem{neumann1992recognition}
Peter~M Neumann and Cheryl~E Praeger.
\newblock A recognition algorithm for special linear groups.
\newblock {\em Proc. London Math. Soc}, 65(432):555--603, 1992.

\bibitem{neunhoffer2006data}
Max Neunh{\"o}ffer and {\'A}kos Seress.
\newblock A data structure for a uniform approach to computations with finite
  groups.
\newblock In {\em Proceedings of the 2006 international symposium on Symbolic
  and algebraic computation}, pages 254--261. ACM, 2006.

\bibitem{niemeyer2012abundant}
Alice~C Niemeyer, Tomasz Popiel, and Cheryl~E Praeger.
\newblock Abundant p-singular elements in finite classical groups.
\newblock {\em arXiv preprint arXiv:1205.1454}, 2012.

\bibitem{niemeyer2010estimating}
Alice~C Niemeyer and Cheryl~E Praeger.
\newblock Estimating proportions of elements in finite groups of {L}ie type.
\newblock {\em Journal of Algebra}, 324(1):122--145, 2010.

\bibitem{obrien2006towards}
EA~O’Brien.
\newblock Towards effective algorithms for linear groups.
\newblock In {\em Finite Geometries, Groups, and Computation: Proceedings of
  the Conference ``Finite Geometries, Groups, and Computation'', Pingree Park,
  Colorado, USA, September 4-9, 2004}, page 163. Walter de Gruyter, 2006.

\end{thebibliography}
\end{document}